\documentclass[12pt]{article}
\usepackage{amsmath,amsthm,amssymb}
\usepackage{pifont}
\usepackage{amsfonts}
\usepackage{mathrsfs}
\usepackage{bm}
\usepackage{mathtools}
\usepackage{graphicx}
\usepackage{enumerate}
\usepackage{setspace}
\usepackage{multirow}
\usepackage{booktabs}

\usepackage{enumitem}
\setlist[enumerate, 1]{label = (\roman*), font = \upshape}
\setlist[itemize, 2]{label = {$\circ$}}

\usepackage[pdfencoding=auto,colorlinks,linkcolor=red,anchorcolor=blue,citecolor=green]{hyperref}
\usepackage{enumitem} 

\usepackage[capitalise]{cleveref}

    \usepackage[margin = 1in]{geometry}
 
 \newcommand{\email}[1]{\href{mailto:#1}{\nolinkurl{#1}}}

\usepackage{fouridx}
\newcommand*{\tran}[1]{\,\fourIdx{t\!}{}{}{}{#1}}

\theoremstyle{plain}
\newtheorem{theorem}{Theorem}[section]
\newtheorem{corollary}[theorem]{Corollary}
\newtheorem{lemma}[theorem]{Lemma}

\theoremstyle{definition}
\newtheorem{definition}[theorem]{Definition}

\newtheorem{remark}{Remark}

\numberwithin{equation}{section}

\makeatletter

\@addtoreset{equation}{section}
\makeatother

\def\~#1{\widetilde{#1}}

\DeclareMathOperator{\C}{\mathbb{C}}
\DeclareMathOperator{\Z}{\mathbb{Z}} 
\DeclareMathOperator{\N}{\mathbb{N}}
\DeclareMathOperator{\F}{\mathbb{F}}

\DeclareMathOperator{\Gr}{\operatorname{Gr}}

\usepackage{tikz,xcolor,hyperref}

\definecolor{lime}{HTML}{A6CE39}
\DeclareRobustCommand{\orcidicon}{
\begin{tikzpicture}
\draw[lime, fill=lime] (0,0)
circle[radius=0.16]
node[white]{{\fontfamily{qag}\selectfont \tiny \.{I}D}}; 

\end{tikzpicture}
\hspace{-2mm}
}

\foreach \x in {A, ..., Z}{%
\expandafter\xdef\csname orcid\x\endcsname{\noexpand\href{https://orcid.org/\csname orcidauthor\x\endcsname}{\noexpand\orcidicon}}
}

\usepackage{fouridx}

\begin{document}

\title{
Trivariate $P$- and $Q$-polynomial structures for interior shells of the Grassmann scheme}
\author{Yan Zhu \hspace{-3mm}\orcidA{}}
\renewcommand{\thefootnote}{\fnsymbol{footnote}}
 \date{}
 \maketitle
\footnotetext{School of mathematics, University of Shanghai for Science and Technology,
Shanghai 200093, China.}
\footnotetext{E-mail address: \email{zhuyan@usst.edu.cn}}

\begin{abstract}
Association schemes induced on the shells of the Grassmann scheme $\Gr_q(N,a)$ form an important family in the study of multivariate $P$- and $Q$-polynomial structures.
For $a<N/2$, the scheme induced on the furthest shell $X_a$ is the known bivariate $P$- and $Q$-polynomial scheme arising from attenuated spaces. 
In this paper, we prove that, for every $1\leq i\leq a-1$, the association scheme induced on the interior shell $X_i$ is both trivariate $P$-polynomial and trivariate $Q$-polynomial. 
The two structures share a common index set consisting of the lattice points of a bounded non-simplicial polyhedron. We establish the $P$-polynomial property through combinatorial interpretations of the intersection numbers. 
Using Dunkl’s coordinate description of $X_i$ and the addition theorem, together with recurrence relations for the $q$-Hahn and $q$-Krawtchouk polynomials, we obtain the required vanishing and nonvanishing conditions for the Krein parameters, and hence the $Q$-polynomial structure. 
Motivated by this index set, we extend the notion of an $A_M$-Leonard pair from simplicial to polyhedral domains and show that every interior shell carries a polyhedral $A_3$-Leonard pair.
\end{abstract} 

\section{Introduction}  
The study of $P$- and $Q$-polynomial association schemes is one of the central problems in algebraic combinatorics.  
In the classical univariate case, the $P$-polynomial property corresponds to distance-regular graphs, while the $Q$-polynomial property gives a dual ordering of the primitive idempotents, which was systematically investigated in Bannai--Bannai--Ito--Tanaka \cite{BBIT2021} and Brouwer--Cohen--Neumaier \cite{BCN-1989}.  
The corresponding first and second eigenvalues are expressed in terms of univariate orthogonal polynomials, see Leonard \cite{Leonard-1982} and Terwilliger \cite{Terwilliger-2021}.  
Related polynomials also appear in the harmonic analysis of finite groups studied by Ceccherini-Silberstein--Scarabotti--Tolli \cite{CST-2008}.  
Multivariate orthogonal polynomials and spherical functions associated with finite groups and Gelfand pairs have also been studied extensively; see more details, for example, in \cite{DX-2014,GR2007,IT-2012,Mizukawa2004,MT2004,Scarabotti2011}.  
A natural question is to identify association schemes whose first or second eigenvalues are given by such multivariate orthogonal polynomials.

The first systematic definitions of bivariate $P$- and $Q$-polynomial association schemes of type $(\alpha,\beta)$ were given by Bernard--Cramp\'e--Poulain d'Andecy--Vinet--Zaimi \cite{BCPVZ-2024}. 
They established the bivariate $P$-polynomial property for examples including nonbinary Johnson schemes and schemes arising from isotropic or attenuated spaces. 
Subsequently, Bannai--Kurihara--Zhao--Zhu \cite{BKZZ-2025} introduced multivariate $P$- and/or $Q$-polynomial association schemes with respect to an arbitrary monomial order.  
They obtained further families of examples, including association schemes arising from compositions of Gelfand pairs \cite{CST2006} and generalized Johnson schemes. 
Bernard--Cramp\'e--Vinet--Zaimi--Zhang \cite{BCVZZ-2024} introduced the notion of an $m$-distance-regular graph and showed that it provides a graph interpretation of multivariate $P$-polynomial association schemes.  
Dually, explicit bivariate $Q$-polynomial structures for nonbinary
Johnson schemes and schemes from attenuated spaces were obtained in \cite{BKZZ-2024,BCVZZ-2025,CVZZ-2024}.
This development is also closely related to higher-rank Leonard pairs.
Rank-two constructions related to the Rahman polynomials \cite{IT-2012} and factorized $A_2$-Leonard pairs \cite{CZ-2025} illustrate this connection.
In \cite[Section~7]{BKZZ-2024}, it is shown that a symmetric multivariate $P$- and $Q$-polynomial scheme on a simplicial index domain has the structure of $A_M$-Leonard pairs.
Therefore, explicit examples of multivariate $P$- and $Q$-polynomial association schemes with higher-rank, non-simplicial index domains are of particular interest.

For several multivariate $P$-polynomial association schemes, the relevant spherical functions are known, but the support and nonvanishing conditions required for the corresponding Krein parameters remain difficult to establish; see \cite[Section~5]{BKZZ-2025}. Thus, explicit spherical functions do not by themselves yield a multivariate $Q$-polynomial structure.
In this paper, we study the multivariate polynomial structures and Leonard pairs of the association schemes induced on the interior shells of the Grassmann scheme $\Gr_q(N,a)$.  
After fixing a base $a$-subspace, its stabilizer partitions the vertex set into shells, and the orbitals on each shell form a symmetric association scheme.  
Dunkl \cite{Dunkl-1977} gave a coordinate description of these orbitals and established an addition theorem for the corresponding spherical functions, which provide the main tools for our analysis.
For $a<N/2$, the scheme induced on the furthest shell $X_a$ is the assciation scheme arising from attenuated spaces introduced by Wang--Guo--Li \cite{WGL-2010}, which was proved to bivariate $P$- and $Q$-polynomial \cite{BKZZ-2024,BKZZ-2025,BCVZZ-2025}. 
For the interior shell $X_i$ ($1\leq i\leq a-1$), both the relations and the primitive idempotents of $X_i$ are indexed by three parameters.  
According to Dunkl \cite{Dunkl-1977}, the multiplicity-free decomposition of $L(X_i)$ into
irreducible components is indexed by the integer points in a bounded non-simplicial polyhedron.

The purpose of this paper is to show that the association scheme induced on every interior Grassmann shell is both trivariate $P$-polynomial and trivariate $Q$-polynomial.  
The $P$-polynomial property follows from the combinatorial interpretation of the intersection numbers, while the $Q$-polynomial property is obtained by the calculation of the Krein parameters using the spherical functions.
The index set for both structures consists of the integer points of a non-simplicial three-dimensional polyhedron.  
Thus, unlike the furthest shell giving an assocaition scheme from attenuated spaces, whose index domain is bivariate, every interior shell is genuinely trivariate. 
This leads us to extend the notion of an $A_M$-Leonard pair to polyhedral index domains and to show that every interior Grassmann shell carries a polyhedral $A_3$-Leonard pair.
The paper is organized as follows.  
\cref{sect:minority} reviews the relations and the eigenmatrices of the association scheme induced on each shell of the Grassmann scheme, as well as Dunkl's description and the addition theorem.  
\cref{sect:turkey} recalls multivariate $P$- and $Q$-polynomial association schemes and their equivalent characterizations.  
In \cref{sect:criticism}, we prove the trivariate
$P$-polynomial structure through intersection numbers, while we establish the trivariate $Q$-polynomial structure through Krein parameters in \cref{sect:profit}.  
Finally, in \cref{sect:commemorate}, we introduce the polyhedral $A_M$-Leonard pair and prove that every interior shell of the Grassmann scheme carries a polyhedral $A_3$-Leonard pair.

\section{Association schemes arising from Grassmann shells}
\label{sect:minority}
In this section, we review the symmetric association scheme induced on the shells of the Grassmann scheme $\Gr_q(N,a)$ as well as its eigenmatrices, following the notation in \cite{Dunkl-1977}.

Let $\F_q$ be the finite field of order $q$, and write
$\F_q^N=\F_q^a\oplus\F_q^b$, where $N=a+b$ and $1\leq a\leq b$.
For $0\leq k\leq N$, let $\mathcal L_k$ be the set of all $k$-dimensional
subspaces of $\F_q^N$, and denote $\mathcal L=\bigsqcup_{k=0}^N\mathcal L_k$.
Let $L(\mathcal L)$ be the space of complex-valued functions on $\mathcal L$.
For $\zeta\in\mathcal L$, define $\hat\zeta\in L(\mathcal L)$ by
\[
 \hat\zeta(\eta)=
 \begin{cases}
  1,&\text{if }\zeta\subseteq\eta,\\
  0,&\text{otherwise}.
 \end{cases}
 \qquad (\eta\in\mathcal L).
\]
Let $P_k$ be the span of $\{\hat\zeta\mid\zeta\in\mathcal L_k\}$ and
define $\partial:P_k\rightarrow P_{k-1}$ by
\[
 \partial\hat\zeta=
 \sum_{\substack{\eta\subseteq\zeta\\\dim\eta=k-1}}\hat\eta.
\]
Then $\dim P_k={N\brack k}_q$.
Set $P_{-1}=0$ and $\partial P_0=0$.
For $0\leq k\leq\lfloor N/2\rfloor$, let $V_k=P_k\cap\ker\partial$.
Let $G=GL(N,q)$ act naturally on $\mathcal L$ and also on $L(\mathcal L)$.  
By \cite[Theorem~3.6]{Dunkl-1977}, the spaces $V_0,V_1,\ldots,V_{\lfloor N/2\rfloor}$ are pairwise nonisomorphic
irreducible $G$-modules.

Let $X=\mathcal L_a$ be the set of all $a$-dimensional subspaces of
$\F_q^N$.
For $0\leq i \leq a$, define the relation $R_i$ on $X$ as follows:
\[
 (x,y)\in R_i\quad\Longleftrightarrow\quad
 \dim(x\cap y)=a-i.
\]
The pair $(X=\mathcal L_a,\{R_i\}_{i=0}^a)$ is the \emph{Grassmann
association scheme} denoted as $\Gr_q(N,a)$, also called the $q$-Johnson scheme $J_q(N,a)$, see
\cite{BBIT2021} for more details.

\subsection{Relations on a fixed shell $X_i$}
Fix an $a$-dimensional subspace $x_0=
 \left\{
 \tran{(\tran{\mathfrak{x}_1},\tran{\mathfrak{o}})}
 \mid \mathfrak{x}_1\in\F_q^a
 \right\}.$
 For $0\leq i\leq a$, the $i$-th shell of $\Gr_q(N,a)$ with respect to $x_0$ is
\[X_i=
 \{x\in X\mid (x_0,x)\in R_i\}
 =
 \{x\in X\mid \dim(x\cap x_0)=a-i\}.\]
Notice that $X_i=X_{a-i,i}$ in the notation of Dunkl~\cite{Dunkl-1977}.  

Let $H$ be the maximal parabolic subgroup of $G$ fixing the subspace
$x_0$.  
Dunkl showed that $H$ acts transitively on $X_i$.
Thus $X_i\simeq H/H_{a-i,i}$, where $H_{a-i,i}$ is the stabilizer in $H$ of a
point of $X_i$.
By \cite[Theorem~4.1]{Dunkl-1977}, a point of $X_i$ can be represented
by a triple $(U,W,T)$, where  
\[U\leq\F_q^a, \quad W\leq\F_q^b, \quad \dim U=a-i, \quad  \dim W=i,\quad
 T:W\longrightarrow\F_q^a/U. \]
Let $P=(U,W,T)$ and $P'=(U',W',T')$ be two
points of $X_i$.  
Let $D_{P,P'}=W\cap W'$, and define two natural quotient maps
\[
 \pi:\F_q^a/U\longrightarrow\F_q^a/(U+U'),\qquad 
 \pi':\F_q^a/U'\longrightarrow\F_q^a/(U+U').
\] 
Hence the restrictions of $T$ and $T'$ induce two maps from
$D_{P,P'}$ to $\F_q^a/(U+U')$ as follows
\[
 \Theta_{P;P'}=\pi\circ T|_{D_{P,P'}}, \qquad 
 \Theta_{P';P}=\pi'\circ T'|_{D_{P,P'}}.
\]
 Denote $v(P,P')=\dim(U\cap U')$, $w(P,P')=\dim(W\cap W')$ and $z(P,P')=\operatorname{rank}
 \bigl(\Theta_{P;P'}-\Theta_{P';P}\bigr)$.
Define the relations on $X_i$ indexed by $(v,w,z)$ as follows:
\begin{equation}\label{eqn:return}
 (P,P')\in R_{v,w,z}^{(i)}
 \quad\Longleftrightarrow\quad
 \bigl(v(P,P'),w(P,P'),z(P,P')\bigr)=(v,w,z).
\end{equation}
Denote $c_i(v)=2i-a+v$.
Then the indices of the relations on $X_i$ form the following set
\begin{equation}\label{eqn:clinic}
 \mathcal I_i=
 \left\{
 (v,w,z)\in\mathbb N^3:
 \begin{array}{l}
 \max\{0,a-2i\}\leq v\leq a-i,\\
 \max\{0,2i-b\}\leq w\leq i,\\
 0\leq z\leq\min\{w,c_i(v)\}
 \end{array}
 \right\}.
\end{equation}
By \cite[Theorem~5.6]{Dunkl-1977}, the relations
$R^{(i)}_{v,w,z}$  are precisely the $H$-orbitals on $X_i$.  
Dunkl showed that the
permutation representation of $H$ on $X_i$ is multiplicity-free.
Moreover, we can check that the relation $R_{v,w,z}^{(i)}$  is symmetric.
Hence $(X_i,\{R^{(i)}_{v,w,z}\}_{(v,w,z)\in \mathcal I_i})$
forms a symmetric (also commutative) association scheme.

\begin{remark}
The parameters in \cite[Theorem~5.6]{Dunkl-1977} are obtained by fixing a point $P_0=(U_0,W_0,0)$ in $X_i$.
For $P=(U,W,T)\in X_i$, we have $\Theta_{P_0;P}=0$ and
\[
 \Theta_{P;P_0}
 =
 \left(
 W\cap W_0\longrightarrow W
 \xrightarrow{T}\F_q^a/U
 \longrightarrow\F_q^a/(U+U_0)
 \right).
\]
Consequently, $(P_0,P)\in R_{v,w,z}^{(i)}$ if and only if
\[
 v=\dim(U\cap U_0),\quad
 w=\dim(W\cap W_0), \quad
 z=
 \operatorname{rank}\bigl(
 \Theta_{P_0;P}-\Theta_{P;P_0}
 \bigr)
 =
 \operatorname{rank}\Theta_{P;P_0}.
\]
\end{remark}
Throughout the paper, we keep Dunkl's notation $V_{m,n,r}$ for the irreducible $H$-modules.
\begin{lemma}[{\cite[Corollary~4.20]{Dunkl-1977}}]
The irreducible $H$-module $V_{m,n,r}$ occurs in the restriction of the
$k$-th irreducible $G$-module $V_k$ if and only if
\[2m+r\leq a,\qquad
 2n+r\leq b,\qquad
 k-b\leq m-n\leq a-k,\qquad
 m+n+r\leq k.\]
\end{lemma}

\begin{lemma}[{\cite[Theorem~5.7]{Dunkl-1977}}]
The space of functions on $X_i$ has the multiplicity-free decomposition
\[L(X_i)
 =
 \bigoplus_{(m,n,r)\in\mathcal D_i^*}
 \rho_{a-i,i}V_{m,n,r},\]
where $\rho_{a-i,i}f=f|_{X_i}$ and
\[\mathcal D_i^*
 =
 \left\{
 (m,n,r)\in\mathbb N^3\mid
 0\leq r\leq i,
 0\leq m\leq\min\{a-i,i-r\},
 0\leq n\leq\min\{i-r,b-i\}
 \right\}.\]
Moreover, $\rho_{a-i,i}V_{m,n,r}\cong V_{m,n,r}$ for every
$(m,n,r)\in\mathcal D_i^*$.
\end{lemma}

Let $E_{m,n,r}^{(i)}$ denote the primitive idempotent corresponding to the
irreducible component $\rho_{a-i,i}V_{m,n,r}\subseteq L(X_i).$
In particular, $E_{0,0,0}^{(i)}$ is the orthogonal projection onto the
constant functions on $X_i$.

\subsection{Spherical functions and eigenmatrices}

In this subsection, we give the eigenmatrices of the association scheme on $X_i$ from
spherical functions discussed in \cite[Theorems~5.6 and~5.7]{Dunkl-1977}.
Recall that the relations of the association
scheme on $X_i$ are the $H$-orbitals
\[
 R^{(i)}_{v,w,z}
 =
 \left\{
 (P,P')\in X_i\times X_i:
 \bigl(v(P,P'),w(P,P'),z(P,P')\bigr)=(v,w,z)
 \right\}.
\]
Denote $\tau_0=(a-i,i,0)$.
Then the diagonal relation is $R^{(i)}_{\tau_0}$.
For later use, define the number of rank-$t$ maps from $\F_q^A$ to
$\F_q^B$  by
\[\mathsf R_q(A,B;t)
 =
 \prod_{\ell=0}^{t-1}
 \frac{(q^A-q^\ell)(q^B-q^\ell)}{q^t-q^\ell}, \quad 0\leq t\leq\min\{A,B\},
 \qquad
 \mathsf R_q(A,B;0)=1.\]
Taking $x=y=a-i$ in the counting number $N(x,v,w,z;y)$ after \cite[Theorem 6.2]{Dunkl-1977}, and using \cite[Proposition~4.4]{Dunkl-1977} for the number of rank-$z$ maps, we obtain the valency of the relation indexed by $\tau=(v,w,z)\in \mathcal I_i$ as follows
\[\begin{aligned}
 k^{(i)}_{(v,w,z)}
  ={}&{a-i\brack v}_q{i\brack a-i-v}_q q^{(a-i-v)^2}
 {i\brack w}_q{b-i\brack i-w}_q q^{(i-w)^2}
 \notag\\
 &\quad\times
 \mathsf R_q(w,c_i(v);z)
 q^{w(a-i-v)+i(i-w)}. 
 \end{aligned}\]
 Assume ${L\brack -1}_q=0$, and set $\Delta_q(L,t)=
 {L\brack t}_q-{L\brack t-1}_q$.
 By \cite[Proposition~4.15]{Dunkl-1977}, the multiplicity of the primitive
idempotent indexed by
$\lambda=(m,n,r)\in\mathcal D_i^*$ equals
\[d_{m,n,r}
 =
 \dim V_{m,n,r}
 =
 \mathsf R_q(a,b;r)
 \Delta_q(a-r,m)
 \Delta_q(b-r,n).\]
Let $E_m(A,B,C,x;q^{-1})$ and $K_r(z;A,B;q)$ be the $q$-Hahn and
$q$-Krawtchouk polynomials in the notation of
\cite{Dunkl-1977}.  
For $ \lambda=(m,n,r)\in\mathcal D_i^*$ and $\tau=(v,w,z)\in\mathcal I_i$
such that $r\leq\min\{w,c_i(v)\}$, define
\[\Phi_{m,n,r}^{(i)}(v,w,z)
 =
 E_m(a-i,i-r,a-i,v;q^{-1})
 E_n(i-r,b-i,i-r,w-r;q^{-1})
 K_r(z;w,c_i(v);q).\]
Assume $\Phi_{m,n,r}^{(i)}(v,w,z)=0$ if
$r>\min\{w,c_i(v)\}$.  Moreover,
$\Phi_{m,n,r}^{(i)}(\tau_0)\neq0$.
We define the normalized spherical function so that $\phi_{m,n,r}^{(i)}(\tau_0)=1$ by
\begin{equation}\label{eqn:calculation}
\phi_{m,n,r}^{(i)}(v,w,z)
 =
 \frac{\Phi_{m,n,r}^{(i)}(v,w,z)}
 {\Phi_{m,n,r}^{(i)}(\tau_0)}.
\end{equation}
Let $A_\tau$ denote the adjacency matrix of $R_\tau^{(i)}$, and let
$E_\lambda=E_{m,n,r}^{(i)}$ be the primitive idempotent corresponding to
$\rho_{a-i,i}V_{m,n,r}$, where $d_\lambda=d_{m,n,r}$.  We use the convention
\[
 A_\tau E_\lambda=P_\tau(\lambda)E_\lambda,
 \qquad
 E_\lambda
 =
 \frac1{|X_i|}
 \sum_{\tau\in\mathcal I_i}Q_\lambda(\tau)A_\tau.
\]
We call $P_\tau(\lambda)$ and $Q_\lambda(\tau)$ the first and second eigenvalues of the association scheme, respectively.

\begin{theorem}
\label{thm:hallway}
For $\tau=(v,w,z)\in\mathcal I_i$ and
$\lambda=(m,n,r)\in\mathcal D_i^*$, the first and second eigenvalues of
the association scheme induced on $X_i$ are
\begin{equation}\label{eqn:finish}
P_\tau(\lambda)=
 k^{(i)}_\tau \cdot \phi_\lambda^{(i)}(\tau), \qquad
 Q_\lambda(\tau)
 =
 d_\lambda \cdot \phi_\lambda^{(i)}(\tau).    
\end{equation}
\end{theorem}

\begin{proof}
By \cite[Theorems~5.6 and~5.7]{Dunkl-1977}, the pair $(H,H_{a-i,i})$ is a finite Gelfand pair, and up to a nonzero scalar,
$\Phi_{m,n,r}^{(i)}$  is the zonal spherical
function associated with the irreducible component $\rho_{a-i,i}V_{m,n,r}\subseteq L(X_i).$
Hence $\phi_\lambda^{(i)}$ is its normalization at the diagonal relation.  
The adjacency operator $A_\tau$ acts on this component by the scalar $k^{(i)}_\tau \cdot \phi_\lambda^{(i)}(\tau)$, which gives the value of $P_\tau(\lambda)$. 
The value of $Q_\lambda(\tau)$ follows from the identity
$\frac{Q_\lambda(\tau)}{d_\lambda}=
\frac{P_\tau(\lambda)}{k^{(i)}_\tau}$, see
\cite[Theorem~2.2]{BBIT2021}.
\end{proof}

\section[Multivariate P- or Q-polynomial association schemes]{Multivariate $P$- or $Q$-polynomial association scheme}\label{sect:turkey}
 In this section, we mainly recall the concepts of multivariate $P$- and/or $Q$-polynomial association schemes proposed in \cite{BKZZ-2025}. 
We first recall the monomial order.
\begin{definition}[{\cite[Definition~2.2]{BKZZ-2025}}] \label{def:genetic}
    A \emph{monomial order} $\le$ on $\C [x_1, x_2, \ldots , x_\ell]$
    is a relation on the set of monomials $x_1^{n_1}x_2^{n_2}\cdots x_\ell^{n_\ell}$ satisfying:
    \begin{enumerate}
        \item $\le$ is a total order;
        \item  for monomials $u,v,w$, if $u\le v$, then $w u \le w v$; 
        \item$\le$ is a well-ordering, i.e., any nonempty subset of the set of monomials
        has a minimum element under $\le$.
    \end{enumerate}
\end{definition}
Denote $\N^\ell:=\{(n_1,n_2,\ldots ,n_\ell) \mid \text{$n_i$ are nonnegative integers}\}$.
For $\alpha=(n_1,n_2,\ldots , n_\ell)\in \N^\ell$
and $\bm{x}=(x_1,x_2,\ldots , x_\ell)$,
we write the monomial $x_1^{n_1}x_2^{n_2}\cdots x_\ell^{n_\ell}$
as $\bm{x}^\alpha$.
Denote $o:=(0,0,\ldots,0)\in \N^\ell$.
Let $\epsilon_1,\epsilon_2,\ldots,\epsilon_\ell$ be the standard basis
vectors of $\N^\ell$.  We use the same notation $\le$ to denote the order on
$\N^\ell$.
The conditions (ii) and (iii) in \cref{def:genetic} mean that, for $\alpha,\beta,\gamma\in \N^\ell$,  $\alpha \le \beta$ implies that $\alpha+\gamma\le \beta+\gamma$, and $o\le \alpha$ for any $\alpha \in \N^\ell$.

\begin{definition}
Consider the coordinatewise partial order $\preceq_{cw}$ on $\N^\ell$: 
\[\text{for~} \alpha=(n_1,\ldots,n_\ell), \beta=(m_1,\ldots,m_\ell), \text{we~write~}\beta\preceq_{cw}\alpha \text{~if~} m_i\le n_i \text{~for~every~}i.\]
We call a subset $\mathcal D\subset\N^\ell$ a {\emph{down-set}} with respect to $\preceq_{cw}$ if
$\alpha\in\mathcal D$ and $\beta\preceq_{cw}\alpha$ imply $\beta\in\mathcal D$.
\end{definition}

 \begin{definition}[{\cite[Definition~2.7]{BKZZ-2025}}]\label{def:twin}
    Let $\mathcal{D}\subset \N^\ell$ be a down-set with respect to $\preceq_{cw}$ containing $\epsilon_1,\ldots,\epsilon_\ell$, and let $\le$ be a monomial order on $\N^\ell$. 
    A commutative association scheme $\mathfrak{X}=(X,\mathcal{R})$ is called \emph{$\ell$-variate $P$-polynomial}
    on the domain $\mathcal{D}$ with respect to $\le$
    if the following two conditions are satisfied:
    \begin{enumerate}[label=$(\roman*)$]
        \item there exists a relabeling $\{A_\alpha\}_{\alpha\in\mathcal D}$
        of the adjacency matrices of $\mathfrak{X}$
        such that, for $\alpha\in \mathcal{D}$,
        \[A_\alpha=v_\alpha(A_{\epsilon_1},A_{\epsilon_2},\ldots ,A_{\epsilon_\ell}), \]
        where $v_\alpha(\bm{x})$ is an $\ell$-variate polynomial
        of multidegree $\alpha$ with respect to $\le$
        and all monomials $\bm{x}^\beta$ in $v_\alpha(\bm{x})$ satisfy $\beta \in \mathcal{D}$;
        \item for $i=1,2,\ldots ,\ell$ and $\alpha=(n_1,n_2,\ldots,n_\ell)\in \mathcal{D}$,
        the product
        $A_{\epsilon_i}\cdot A_{\epsilon_1}^{n_1}A_{\epsilon_2}^{n_2}\cdots A_{\epsilon_\ell}^{n_\ell}$
        is a linear combination of
        \[
         \{A_{\epsilon_1}^{m_1}A_{\epsilon_2}^{m_2}\cdots A_{\epsilon_\ell}^{m_\ell} \mid \beta =(m_1,m_2,\ldots ,m_\ell)\in \mathcal{D},\  \beta \le  \alpha +\epsilon_i\}.   
        \]
    \end{enumerate}
\end{definition}

\begin{definition}[{\cite[Definition~2.8]{BKZZ-2025}}]\label{def:expertise}
    Let $\mathcal{D}^\ast \subset \N^\ell$
    be a down-set with respect to $\preceq_{cw}$ containing $\epsilon_1,\ldots,\epsilon_\ell$,  
    and let $\le$ be a monomial order on $\N^\ell$.
    A commutative association scheme $\mathfrak{X}=(X,\mathcal{R})$
    is called \emph{$\ell$-variate $Q$-polynomial}
    on the domain $\mathcal{D}^\ast$ with respect to $\le$
    if the following two conditions are satisfied:
    \begin{enumerate}[label=$(\roman*')$]
        \item there exists a relabeling $\{E_\alpha\}_{\alpha\in\mathcal D^\ast}$
        of the primitive idempotents
        such that, for $\alpha\in \mathcal{D}^\ast$,
        \[
            |X| E_\alpha=v^\ast_\alpha(|X| E_{\epsilon_1},|X| E_{\epsilon_2},\ldots ,|X| E_{\epsilon_\ell})
            \ \text{(under the Hadamard product),}
        \]
        where $v^\ast_\alpha(\bm{x})$ is an $\ell$-variate polynomial
        of multidegree $\alpha$ with respect to $\le$
        and all monomials $\bm{x}^\beta$ in $v^\ast_\alpha(\bm{x})$ satisfy $\beta \in \mathcal{D}^\ast$;
        \item for $i=1,2,\ldots ,\ell$ and $\alpha=(n_1,n_2,\ldots,n_\ell)\in \mathcal{D}^\ast$,
        the product
        $E_{\epsilon_i}\circ E_{\epsilon_1}^{\circ n_1}\circ E_{\epsilon_2}^{\circ n_2}\circ \cdots  \circ E_{\epsilon_\ell}^{\circ n_\ell}$
        is a linear combination of
        \[
            \{E_{\epsilon_1}^{\circ m_1}\circ E_{\epsilon_2}^{\circ m_2}\circ \cdots  \circ E_{\epsilon_\ell}^{\circ m_\ell} \mid \beta=(m_1,m_2,\ldots ,m_\ell)\in \mathcal{D}^\ast,\  \beta\le  \alpha+\epsilon_i\}.   
        \]
    \end{enumerate}
\end{definition}
Let $\mathfrak X=(X,\mathcal R)$ be a commutative association scheme with
adjacency matrices $\{A_\alpha\}_{\alpha\in\mathcal D}$ and primitive
idempotents $\{E_\alpha\}_{\alpha\in\mathcal D^\ast}$.  The intersection
numbers $p_{\alpha,\beta}^{\gamma}$ and Krein numbers
$q_{\alpha,\beta}^{\gamma}$ are defined by
\[A_\alpha A_\beta
 =
 \sum_{\gamma\in\mathcal D}
 p_{\alpha,\beta}^{\gamma}A_\gamma, \qquad 
 E_\alpha\circ E_\beta
 =
 \frac{1}{|X|}\sum_{\gamma\in\mathcal D^\ast}
 q_{\alpha,\beta}^{\gamma}E_\gamma.\]
The next two lemmas characterize the $\ell$-variate $P$- and
$Q$-polynomial properties in terms of intersection numbers and Krein
numbers, respectively.
\begin{lemma}[{\cite[Proposition~2.15]{BKZZ-2025}}]\label{lem:systematic}
Let $\mathcal{D}\subset \N^\ell$ be a down-set with respect to $\preceq_{cw}$ containing $\epsilon_1,\ldots,\epsilon_\ell$, and let $\le$ be a monomial order on $\N^\ell$.
Let $\mathfrak{X}$
be a commutative association scheme with the adjacency matrices $\{A_\alpha\}_{\alpha\in \mathcal{D}}$.
Then the following two statements are equivalent:
\begin{enumerate}
    \item $\mathfrak{X}$ is
    an $\ell$-variate $P$-polynomial association scheme 
    on $\mathcal{D}$ with respect to  $\le$;
    \item 
    the intersection numbers satisfy,
    for each $i=1,2,\ldots ,\ell$ and each $\alpha \in \mathcal{D}$,
    $p^\beta_{\epsilon_i, \alpha} \neq 0$ for $\beta\in \mathcal{D}$ implies $\beta \le \alpha+\epsilon_i$. 
    Moreover, if $\alpha+\epsilon_i\in \mathcal{D}$, then $p^{\alpha+\epsilon_i}_{\epsilon_i, \alpha} \neq 0$ holds.
\end{enumerate}    
\end{lemma}

\begin{lemma}[{\cite[Proposition~2.16]{BKZZ-2025}}]
    \label{prop:substitute}
    Let $\mathcal{D}^\ast\subset \N^\ell$ be a down-set with respect to $\preceq_{cw}$ containing $\epsilon_1,\ldots,\epsilon_\ell$, and let $\le$ be a monomial order on $\N^\ell$.
Let $\mathfrak{X}$ be a commutative association scheme
    with the primitive idempotents $\{E_\alpha\}_{\alpha\in \mathcal{D}^\ast}$.
    Then the  following two statements are equivalent:
    \begin{enumerate}[label=$(\mathrm{\roman*})$]
        \item $\mathfrak{X}$ is
        an $\ell$-variate $Q$-polynomial association scheme 
        on $\mathcal{D}^\ast$ with respect to $\le$;
        \item 
         the Krein numbers satisfy,
        for each $i=1,2,\ldots ,\ell$ and each $\alpha \in \mathcal{D}^\ast$,
        $q^\beta_{\epsilon_i, \alpha} \neq 0$ for $\beta\in \mathcal{D}^\ast$ implies $\beta \le \alpha+\epsilon_i$. 
        Moreover, if $\alpha+\epsilon_i\in \mathcal{D}^\ast$, then $q^{\alpha+\epsilon_i}_{\epsilon_i, \alpha} \neq 0$ holds.
    \end{enumerate}
\end{lemma}

\begin{remark}
 The shell $X_0$ of the Grassmann scheme $\operatorname{Gr}_q(a+b,a)$ is a singleton. The furthest shell $X_a$ consists of all $a$-dimensional subspaces that intersect $x_0$ trivially, and the association scheme induced on $X_a$ is the one from attenuated spaces. 
 When $a<b$, this scheme is bivariate $P$-polynomial and bivariate $Q$-polynomial, as proved in \cite{BKZZ-2025,BKZZ-2024,BCVZZ-2025}. 
 When $a=b$, the scheme induced on $X_a$ is isomorphic to the bilinear-forms scheme $H_q(a,a)$, which is univariate $P$-polynomial and univariate $Q$-polynomial.   
\end{remark}

Therefore, in the rest of the paper, we focus on the interior shells $X_i$ of the Grassmann scheme $\Gr_q(a+b,a)$, where $1\leq i\leq a-1$.

To obtain the trivariate $P$- and $Q$-polynomial structures, we introduce a special order $\le_{ijk}$ as follows.
Define the lexicographic order by $\alpha\le_{\mathrm{lex}}\beta$ if
$\alpha=\beta$ or the leftmost nonzero entry of
$\alpha-\beta\in\Z^\ell$ is negative.
For triple case, we define the graded lexicographic order $\le_{ijk}$ on $\N^3$ with a permutation $(i,j,k)$ of $(1,2,3)$ as follows:
$(x_1,x_2,x_3)\le_{ijk}(y_1,y_2,y_3)$ if either
$x_1+x_2+x_3<y_1+y_2+y_3$, or the two sums are equal and
$(x_i,x_j,x_k)\le_{\mathrm{lex}}(y_i,y_j,y_k)$.

\section[The trivariate P-polynomial structure]{The trivariate $P$-polynomial structure}\label{sect:criticism}
In this section, we prove that each interior shell of the Grassmann scheme is trivariate $P$-polynomial using combinatorial interpretations of the intersection numbers.

Recall that the relations of the association scheme on $X_i$ are indexed by $ \mathcal I_i$. 
We map $(v,w,z)$ to the new coordinates as
$x=a-i-v$, $y=i-w$
and define
\begin{equation}\label{eqn:circulate}
\mathcal D_i
 =
 \left\{
 (x,y,z)\in\mathbb N^3 \mid
 (a-i-x,i-y,z)\in\mathcal I_i
 \right\}.
\end{equation}
Equivalently, we have
\[
 \mathcal D_i=
 \left\{
 (x,y,z)\in\mathbb N^3\mid
 0\leq x\leq\min\{a-i,i\},
 0\leq y\leq\min\{i,b-i\},
 0\leq z\leq i-\max\{x,y\}
 \right\}.
\]
For $\alpha=(x,y,z)\in\mathcal D_i$, put $\tau(\alpha)=(a-i-x,i-y,z)\in\mathcal I_i$. 
We relabel the relations and their adjacency matrices by
\[
 \widehat R_\alpha=R^{(i)}_{\tau(\alpha)},
 \qquad
 \widehat A_\alpha=A_{\tau(\alpha)}.
\]
Thus $\widehat R_{(0,0,0)}$ is the diagonal relation. 
Denote $\widehat k^{(i)}_\alpha=k^{(i)}_{\tau(\alpha)}$.
For $\alpha,\beta,\gamma\in\mathcal D_i$, define the relabelled
intersection numbers $\widehat p_{\alpha,\beta}^{\gamma}$ by $\widehat A_\alpha\widehat A_\beta
 =
 \sum_{\gamma\in\mathcal D_i}
 \widehat p_{\alpha,\beta}^{\gamma}\widehat A_\gamma.$
If we write $p_{\alpha,\beta}^{\gamma}$ for the intersection numbers with
respect to $\{A_\gamma\}_{\gamma\in\mathcal I_i}$, then we have $\widehat p_{\alpha,\beta}^{\gamma}
 =p_{\tau(\alpha),\tau(\beta)}^{\tau(\gamma)}.$

The following result is the main theorem of this section.
\begin{theorem}\label{thm:laboratory}
For $1\leq i\leq a-1$, the association scheme induced on the $i$-th shell $X_i$ of the Grassmann scheme is trivariate $P$-polynomial on $\mathcal D_i$ with respect to $\le_{123}$.
\end{theorem}
Put $\epsilon_1=(1,0,0)$, $\epsilon_2=(0,1,0)$ and $\epsilon_3=(0,0,1)$.
By \cref{lem:systematic}, we need first to prove the following lemma.
\begin{lemma}
\label{lem:plain}
For $1\leq i\leq a-1$ and
$\alpha\in\mathcal D_i$, let $\mathcal N_3^P(\alpha)=\{\alpha+\eta \epsilon_3 \mid \eta\in \{-1,0,1\}\}$ and 
\[\mathcal N_s^P(\alpha)=\{\alpha+\xi\epsilon_s+\eta \epsilon_3\mid \xi,\eta\in \{-1,0,1\}, ~ \xi \eta\leq 0\},\quad s=1,2. \]
Then, for $s=1,2,3$, the intersection number $\widehat p_{\epsilon_s,\alpha}^{\beta} \neq 0$ implies that $
 \beta\in\mathcal N_s^P(\alpha)\cap\mathcal D_i.$
Consequently, $\widehat p_{\epsilon_s,\alpha}^{\beta}\neq0$ implies that $\beta\le_{123}\alpha+\epsilon_s$.
\end{lemma}
\begin{proof}
Any point in
$X_i=X_{a-i,i}$ is represented by a triple $ P=(U,W,T)$ satisfying
\[
 \dim U=a-i,\qquad
 \dim W=i,\qquad
 T: W\longrightarrow\F_q^a/U.
\]
Write $\alpha=(x,y,z)$.  Fix two points $P_0=(U_0,W_0,0)$ and $P=(U,W,T)$ from $X_i$ with $(P_0,P)\in\widehat R_\alpha$.
Then
\[
 x=a-i-\dim(U\cap U_0),
 \qquad
 y=i-\dim(W\cap W_0), 
 \qquad 
 z= \operatorname{rank}(\Theta_{P;P_0}).
\]
For $s\in\{1,2,3\}$, suppose that
$\widehat p_{\epsilon_s,\beta}^{\alpha}\neq0$, where
$\beta=(x',y',z')$.  By the definition of the intersection numbers,
there exists a point $P_1=(U_1,W_1,T_1)\in X_i$ such that $(P_0,P_1)\in\widehat R_{\epsilon_s}$ and $(P_1,P)\in\widehat R_\beta.$
More precisely,
\[
 x'=a-i-\dim(U_1\cap U),
 \quad
 y'=i-\dim(W_1\cap W),
\quad
 z'=
 \operatorname{rank}\left(\Theta_{P_1;P}-\Theta_{P;P_1}
 \right).
\]
We determine the possible values of $\beta$ for each $s$.  
In each of the following three cases, we choose a linear lift
$\widetilde T:W\longrightarrow\F_q^a$ of the map $T$, if necessary.
Moreover, for a subspace $B\leq\F_q^a$,
we use the notation $\pi_B:\F_q^a\longrightarrow\F_q^a/B$ to denote the natural quotient map.\\[0.1in]
{\bf{We first consider $\epsilon_1=(1,0,0)$. }} 
In this case,
\[
 W_1=W_0,
 \qquad
 S_U:=U_0\cap U_1,
 \qquad
 \dim S_U=a-i-1,
\qquad
 \operatorname{Im}T_1
 \subseteq (U_0+U_1)/U_1.
\]
Indeed, the quotient $(U_0+U_1)/U_1$ is the kernel of the natural map $\F_q^a/U_1\longrightarrow\F_q^a/(U_0+U_1)$.
Since the natural map $U_0\longrightarrow (U_0+U_1)/U_1$
is surjective, the map $T_1:W_0\longrightarrow (U_0+U_1)/U_1$ has a linear lift $\widetilde T_1:W_0\longrightarrow U_0.$
Since $W_1=W_0$, we have $y'=y$.  

To analyze the relation between $z$ and $z'$, we represent the three maps $\Theta_{P;P_0}$, $\Theta_{P_1;P}$,  $\Theta_{P;P_1}$ as maps into a common quotient space through suitable lifts and quotient maps.
Let
\[
 D=W\cap W_0,
 \qquad
 \mathcal H=\F_q^a/(U+S_U),
 \qquad
 L_j=(U+U_j)/(U+S_U)\leq\mathcal H
 \quad (j=0,1).
\]
Using the natural quotient map $\pi_{U+S_U}:\F_q^a\longrightarrow\mathcal H$, the lifts $\widetilde T$ and $\widetilde T_1$ induce linear maps
\[
 M=\pi_{U+S_U}\circ\widetilde T|_D:
 D\longrightarrow\mathcal H
,\qquad 
 M_1=\pi_{U+S_U}\circ\widetilde T_1|_D:
 D\longrightarrow\mathcal H.
\]
Since $\operatorname{Im}\widetilde T_1\subseteq U_0$ and $U_0+(U+S_U)=U_0+U$, 
we have
\[
 \operatorname{Im}M_1\subseteq \pi_{U+S_U}(U_0)=
 (U+U_0)/(U+S_U)=L_0.
\]
For $j=0,1$, let $\rho_j:\mathcal H\longrightarrow\mathcal H/L_j$ be the natural quotient map.  
Under the natural identification $\mathcal H/L_j\cong\F_q^a/(U+U_j)$, we have  
\[\Theta_{P;P_0}=\rho_0\circ M,\qquad \Theta_{P_1;P}=\rho_1\circ M_1,\qquad \Theta_{P;P_1}=\rho_1\circ M,\]
 \cref{eqn:return} gives $z=\operatorname{rank}(-\rho_0\circ M)$ and $z'=\operatorname{rank}\bigl(\rho_1 \circ M_1-\rho_1\circ M\bigr)$.
 For $j=0,1$, let $\delta_j=
 \dim(U\cap U_j)-\dim(U\cap S_U)\in\{0,1\}$.
 Then $x'-x=\delta_0-\delta_1$
 and 
\[\dim L_j=\dim(U+U_j)-\dim(U+S_U)=\dim U_j-\dim S_U-\dim(U\cap U_j)+\dim(U\cap S_U)=1-\delta_j.\]
Hence $x'-x\in\{-1,0,1\}$.
\begin{enumerate}[label={\rm(\Roman*)}]
    \item  $x'=x+1$, i.e., $L_0=0$, $\dim L_1=1.$\\
Since $\operatorname{Im}M_1\subseteq L_0$, we have $M_1=0$.
Thus $z=\operatorname{rank}M$ and $z'=\operatorname{rank}(\rho_1\circ M)$.
The map $\rho_1$ is a quotient map with one-dimensional kernel.
Consequently, $ z'\in\{z-1,z\}.$

\item $x'=x-1$, i.e., $\dim L_0=1$ and $L_1=0$.\\
Since $\operatorname{Im}M_1\subseteq L_0=\ker\rho_0$, we have $\rho_0\circ M_1=0$, and hence $\rho_0\circ M=\rho_0\circ(M-M_1).$
Thus $z=\operatorname{rank}(\rho_0\circ(M-M_1))$ and $z'=\operatorname{rank}(M-M_1)$.
As $\rho_0$ has a one-dimensional kernel, it follows that $ z'\in\{z,z+1\}.$
\item $x'=x$, i.e., $\dim L_0=\dim L_1$.\\
 Let $\rho:\mathcal H\longrightarrow
 \mathcal H/(L_0+L_1)$ be the natural quotient map.  
 For $j=0,1$, the natural maps 
$\mathcal H/L_j\longrightarrow
 \mathcal H/(L_0+L_1)$
 have kernels of dimension at most one. 
Hence
\[
 \operatorname{rank}(\rho\circ M)\leq z\leq
 \operatorname{rank}(\rho\circ M)+1,
 \quad
 \operatorname{rank}(\rho\circ(M-M_1))\leq z'\leq
 \operatorname{rank}(\rho\circ(M-M_1))+1.
\]
Since $\operatorname{Im}M_1\subseteq L_0\subseteq L_0+L_1=\ker\rho$, 
we have $\rho\circ M_1=0$, and hence $\rho\circ M=\rho\circ(M-M_1).$
Therefore,  $z'\in\{z-1,z,z+1\}.$
\end{enumerate}
{\bf{We next consider $\epsilon_2=(0,1,0)$.}}    
In this case,
\[
 U_1=U_0,
 \qquad
 S_W:=W_0\cap W_1,
 \qquad
 \dim S_W=i-1.
\]
Since $U_1=U_0$ and the third coordinate of $\epsilon_2$ is zero,
we have $x'=x$ and $T_1|_{S_W}=0$.

We claim that, up to the action of the stabilizer of $P_0$ in $H$, we may assume $P_1=(U_0,W_1,0).$
Choose a lift $\widetilde T_1:W_1\to\F_q^a$ of $T_1$ vanishing on $S_W$. Since $S_W=W_0\cap W_1$, there exists a linear map $S:\F_q^b\to\F_q^a$ such that $S|_{W_0}=0$ and $S|_{W_1}=-\widetilde T_1.$
Define $h_S\in H$ by $h_S(u,w)=(u+S(w),w)$. 
Its induced action on triples is $h_S(U,W,T)=\bigl(U,W,T+\pi_U\circ S|_W\bigr).$
Therefore,
\[
h_S(P_0)
=
\bigl(U_0,W_0,\pi_{U_0}\circ S|_{W_0}\bigr)
=(U_0,W_0,0)=P_0
\]
and 
\[
h_S(P_1)
=
\bigl(U_0,W_1,T_1+\pi_{U_0}\circ S|_{W_1}\bigr)
=
\bigl(U_0,W_1,T_1-\pi_{U_0}\circ\widetilde T_1\bigr)
=(U_0,W_1,0).
\]
Since the relations are $H$-orbitals, this proves the claim.

For $j=0,1$, let $D_j=W\cap W_j$ and define
\[
\Lambda_j
=\pi_{U+U_0}\circ \widetilde T |_{D_j}
:D_j\longrightarrow\F_q^a/(U+U_0).
\]
Since $T_0=T_1=0$, we have $\Theta_{P;P_j}=\Lambda_j$ and $\Theta_{P_j;P}=0,$ and hence $z=\operatorname{rank}\Lambda_0$, 
$z'=\operatorname{rank}\Lambda_1.$
The subspace $S_W$ has codimension one in both $W_0$ and $W_1$. 
Moreover, $D_j\cap S_W=D_0\cap D_1$ and $D_j+S_W\subseteq W_j.$
Hence the dimension formula gives
\[
\dim D_j-\dim(D_0\cap D_1)
=\dim(D_j+S_W)-\dim S_W\leq \dim W_j-\dim S_W=1.
\]
Thus $D_j$ is either $D_0\cap D_1$ or a one-dimensional extension of it.
 Also,
\[
y'-y
=(i-\dim D_1)-(i-\dim D_0)
=\dim D_0-\dim D_1
\in\{-1,0,1\}.
\]
\begin{enumerate}[label={\rm(\Roman*)}]
    \item $y'=y+1$. Hence $D_1=D_0\cap D_1\subseteq D_0$ and $\dim D_0-\dim D_1=1.$
By the definitions of $\Lambda_0$ and $\Lambda_1$, we have $\Lambda_1=\Lambda_0|_{D_1}.$
Therefore, $z'\in\{z-1,z\}$.
\item  $y'=y-1$. 
 Hence $D_0=D_0\cap D_1\subseteq D_1$  and $\dim D_1-\dim D_0=1$.
Since $\Lambda_0=\Lambda_1|_{D_0}$, it follows that $z'\in\{z,z+1\}.$
\item $y'=y$. If $D_0=D_1$, then $\Lambda_0=\Lambda_1$, and hence $z'=z$. Otherwise, both $D_0$ and $D_1$ are one-dimensional extensions of $D_0\cap D_1$. Since $\Lambda_0$ and $\Lambda_1$ agree on $D_0\cap D_1$, their ranks differ by at most one. Consequently, $z'\in\{z-1,z,z+1\}.$
\end{enumerate}
{\bf{Finally, we consider $\epsilon_3=(0,0,1)$.}}
In this case,
\[
 U_1=U_0,
 \quad
 W_1=W_0,
 \quad
 \operatorname{rank}T_1=1,
 \quad
 T_1:W_0\longrightarrow\F_q^a/U_0.
\]
Since $U_1=U_0$ and $W_1=W_0$, we have $ x'=x$ and $y'=y.$ 
Denote $ D=W\cap W_0$ and let
\[\pi:\F_q^a\longrightarrow\F_q^a/(U+U_0), \quad 
\pi_0:\F_q^a/U_0\longrightarrow\F_q^a/(U+U_0)\]
 be two natural quotient maps.
 Define  
\[
 M=\pi\circ\widetilde T|_D:
 D\longrightarrow\F_q^a/(U+U_0), \quad 
 M_1=\pi_0\circ T_1|_D:
 D\longrightarrow\F_q^a/(U+U_0).
\]
Hence, we have $\Theta_{P;P_0}=M$, $\Theta_{P_1;P}=M_1$ and $\Theta_{P;P_1}=M$ so that $z'=\operatorname{rank}(M_1-M)$ and  $z= \operatorname{rank}M$.
Since $\operatorname{rank}T_1=1$, we have
$\operatorname{rank}M_1\leq1$, which gives 
\[
 \left|z'-z\right|
 =
 \left|
 \operatorname{rank}(M_1-M)-\operatorname{rank}M
 \right|
 \leq
 \operatorname{rank}M_1
 \leq1.
\]
Combining the three cases, we show that, for $s=1,2,3$, $\widehat p_{\epsilon_s,\beta}^{\alpha}\neq0$ implies that $\beta\in\mathcal N_s^P(\alpha)\cap\mathcal D_i$.
 Using the identity $\widehat k_\beta^{(i)}\widehat p_{\epsilon_s,\alpha}^{\beta}
=
\widehat k_\alpha^{(i)}\widehat p_{\epsilon_s,\beta}^{\alpha}$, we obtain that   $\widehat p_{\epsilon_s,\alpha}^{\beta}\neq0$ implies that $\beta\in\mathcal N_s^P(\alpha)\cap\mathcal D_i$.  
\end{proof}
We now prove the required positivity result.
\begin{lemma}\label{lem:referral}
For $1\leq i\leq a-1$, $s\in\{1,2,3\}$, we have $\widehat p_{\epsilon_s,\alpha}^{\alpha+\epsilon_s}>0$ whenever
$\alpha,\alpha+\epsilon_s\in\mathcal D_i$.
\end{lemma}
\begin{proof}
Given $P_0=(U_0,W_0,0)$ and $P=(U,W,T)$ with
$(P_0,P)\in\widehat R_{\alpha+\epsilon_s}$, 
we need to find a point $P_1$ satisfying
$(P_0,P_1)\in\widehat R_{\epsilon_s}$ and
$(P_1,P)\in\widehat R_\alpha$.
Let $\alpha=(x,y,z)$ and $D=W\cap W_0$.\\ [0.1in]
\noindent\textbf{Case 1: $s=1$.}\\
Choose $u\in U\setminus U_0$ and a hyperplane
$S_U$ of $U_0$ containing $U\cap U_0$, and set
\[
U_1=S_U+\langle u\rangle,
\qquad W_1=W_0.
\]
We can check that $U_0\cap U_1=S_U,$ $\dim(U\cap U_1)=a-i-x$ and $ U+U_1=U+S_U\subseteq U+U_0.$
Let $\pi:\F_q^a/(U+U_1)\longrightarrow
\F_q^a/(U+U_0)$ be the natural map. 
Since $U_0\cap(U+S_U)=S_U,$
we have
\[
\ker\pi=(U+U_0)/(U+U_1)\cong U_0/S_U,
\]
and hence $\ker\pi$ is one-dimensional. Let
$\overline T:D\longrightarrow\F_q^a/(U+U_1)$
be the map induced by $T|_D$. Then $\operatorname{rank}(\pi\circ\overline T)=\operatorname{rank}(\Theta_{P;P_0})=z.$
Choose a decomposition $\F_q^a/(U+U_1)=\ker\pi\oplus C,$
and write $\overline T=\overline T_{\ker\pi}+\overline T_C,$
where $\operatorname{Im}\overline T_{\ker\pi}\subseteq\ker\pi$ and $
\operatorname{Im}\overline T_C\subseteq C.$
Since $\pi|_C:C\to\F_q^a/(U+U_0)$ is an isomorphism, we have $\operatorname{rank}\overline T_C=z.$
The natural map $(U_0+U_1)/U_1\longrightarrow\ker\pi$ is an isomorphism. 
Thus, choose $T_1|_D$ so that its induced map
equals $\overline T_{\ker\pi}$, and extend it to a map
\[
T_1:W_0\longrightarrow (U_0+U_1)/U_1\subseteq\F_q^a/U_1.
\]
Take $P_1=(U_1,W_0,T_1)$. 
Since $\operatorname{Im}T_1\subseteq (U_0+U_1)/U_1,$
we have $(P_0,P_1)\in\widehat R_{\epsilon_1}$. Moreover,
\[
\Theta_{P_1;P}-\Theta_{P;P_1}
=
\overline T_{\ker\pi}-\overline T
=
-\overline T_C.
\]
Hence $\operatorname{rank}
\bigl(\Theta_{P_1;P}-\Theta_{P;P_1}\bigr)
=
\operatorname{rank}\overline T_C
=
z$.
Therefore, we obtain $(P_1,P)\in\widehat R_\alpha.$

\noindent
{\bf{Case 2: $s=2$. }}  \\
Then $\dim D=i-y-1$. Choose $t\in W\setminus W_0$ and a hyperplane $S_W$ of $W_0$ containing $D$, and set $W_1=S_W\oplus\langle t\rangle.$
Choose $v\in\F_q^a$ such that $T(t)=v+U$, and define
\[
T_1:W_1\longrightarrow\F_q^a/U_0, \quad \text{so~that~} T_1|_{S_W}=0,\quad T_1(t)=v+U_0.
\]
Take $P_1=(U_0,W_1,T_1)$. 
Since $W_0\cap W_1=S_W$ and $T_1|_{S_W}=0,$ we have $(P_0,P_1)\in\widehat R_{\epsilon_2}$.
Let $\overline T$ and $\overline T_1$ be the maps induced by $T$ and $T_1$, respectively, in the common quotient
$\F_q^a/(U+U_0)$. 
Since $W\cap W_1=D\oplus\langle t\rangle,$
we have
\[
(\overline T_1-\overline T)|_D
=-\overline T|_D,
\qquad
(\overline T_1-\overline T)(t)=0.
\]
Hence $z(P_1,P)
=\operatorname{rank}(\overline T_1-\overline T)
=\operatorname{rank}(\overline T|_D)
=z.$
Finally, $\dim(W\cap W_1)=\dim D+1=i-y,$ and therefore $(P_1,P)\in\widehat R_\alpha$.\\[0.1in]
\noindent
{\bf{Case 3: $s=3$. }} \\
Let $F:D\longrightarrow\F_q^a/(U+U_0)$ be the map induced by $T|_D$. Since
$(P_0,P)\in\widehat R_{\alpha+\epsilon_3}$, \cref{eqn:return} gives $\operatorname{rank}F=z+1.$
Choose a rank-one map
\[
F_1:D\longrightarrow\F_q^a/(U+U_0), \quad 
\text{such~that~} \operatorname{rank}(F-F_1)=z.
\]
The natural quotient map $\rho_0:\F_q^a/U_0\longrightarrow
\F_q^a/(U+U_0)$ is surjective.
Choose a rank-one map $T_1:W_0\longrightarrow\F_q^a/U_0$ such that $\left.\rho_0\circ T_1\right|_D=F_1$.
Such a map is obtained by lifting $F_1$ to $D$ and extending the lift by zero to $W_0$.
Take $P_1=(U_0,W_0,T_1).$ 
Since $\operatorname{rank}T_1=1$, we have $(P_0,P_1)\in\widehat R_{\epsilon_3}.$
In the common quotient $\F_q^a/(U+U_0)$, the maps induced on $D$ by $T_1$ and $T$ are $F_1$ and $F$, respectively. Thus $z(P_1,P)=\operatorname{rank}(F_1-F)=z.$
Since $P_0$ and $P_1$ have the same first two components, this yields $(P_1,P)\in\widehat R_\alpha.$

The required point $P_1$ exists in each case, so
$\widehat p_{\epsilon_s,\alpha}^{\alpha+\epsilon_s}>0$ for $s=1,2,3$.
\end{proof}
Combining \cref{lem:plain,lem:referral}, we give the proof of the main result of this section.
\begin{proof}[{Proof of \cref{thm:laboratory}}]
We first verify that $\mathcal D_i$ is a down-set w.r.t. $\preceq_{cw}$. 
Let $(x,y,z)\in\mathcal D_i$ and suppose that
$0\leq x'\leq x,$ $0\leq y'\leq y$ and
$0\leq z'\leq z.$
The upper bounds for $x'$ and $y'$ follow immediately from those
for $x$ and $y$. Moreover, $z'
\leq z
\leq i-\max\{x,y\}
\leq i-\max\{x',y'\}.$
Hence $(x',y',z')\in\mathcal D_i$.
Since $1\leq i\leq a-1$ and $a\leq b$, the domain
$\mathcal D_i$ contains $\epsilon_1,\epsilon_2,\epsilon_3$.

By \cref{lem:plain,lem:referral}, the intersection numbers satisfy $\widehat p_{\epsilon_s,\alpha}^{\beta}\neq0$ implying $\beta\leq_{123}\alpha+\epsilon_s,$
and
$\widehat p_{\epsilon_s,\alpha}^{\alpha+\epsilon_s}>0$
whenever $\alpha+\epsilon_s\in\mathcal D_i$.
Therefore, \cref{lem:systematic} shows that the association scheme on
$X_i$ is trivariate $P$-polynomial on $\mathcal D_i$.      
\end{proof}
\section[The trivariate Q-polynomial structure]{The trivariate $Q$-polynomial structure}
\label{sect:profit}
In this section, we prove that each interior shell of the Grassmann scheme is trivariate $Q$-polynomial by the required vanishing and nonvanishing conditions for the Krein parameters.

The proof is outlined as follows.
We first express the spherical functions $\phi_{m,n,r}^{(i)}(v,w,z)$ as products of two $q$-Hahn polynomials $H_m^{(r)}(v)$, $J_n^{(r)}(w)$ and a $q$-Krawtchouk polynomial $L_r(v,w,z)$.
We then use the recurrence relations for these orthogonal polynomials to obtain expansions for $\phi^{(i)}_{\varepsilon_s}(\tau) \phi^{(i)}_{\alpha}(\tau)$.  
By comparing these expansions with the formula for the Krein parameters, we show that $q_{\epsilon_s,\alpha}^{\beta}=0$ unless $\beta\in\mathcal N_s^Q(\alpha)\cap\mathcal D_i^*$.
Finally, we use leading coefficients of these orthogonal polynomials to show that $q_{\epsilon_s,\alpha}^{\alpha+\epsilon_s}\neq 0$ whenever $\alpha, \alpha+\epsilon_s\in\mathcal D_i^*$.

Throughout the remainder of this section, an orthogonal polynomial, a spherical function, or a coefficient with the index outside its admissible set is regarded as zero. 

For $\lambda,\mu,\nu\in\mathcal D_i^*$, the Krein number is given by the second orthogonality relation
\begin{equation}\label{eqn:overview}
 q_{\lambda,\mu}^{\nu}
 =\frac{1}{|X_i|d_\nu}
 \sum_{\tau\in\mathcal I_i}k^{(i)}_\tau
 Q_\lambda(\tau)Q_\mu(\tau)Q_\nu(\tau),
 \qquad
 |X_i|=q^{i^2}{a\brack i}_q{b\brack i}_q.
\end{equation}
For fixed $i$, we omit the superscript $(i)$ from $\phi_\lambda^{(i)}(\tau)$.
Substitution of $Q_\lambda=d_\lambda\phi_\lambda$ in \cref{thm:hallway} into
\cref{eqn:overview} gives, for
$\alpha,\beta\in\mathcal D_i^*$,
\begin{equation}\label{eqn:deviation}
 q_{\epsilon_s,\alpha}^{\beta}
 =\frac{d_{\epsilon_s}d_\alpha}{|X_i|}
 \sum_{(v,w,z)\in\mathcal I_i}
 k^{(i)}_{(v,w,z)}
 \phi_{\epsilon_s}(v,w,z)
 \phi_\alpha(v,w,z)\phi_\beta(v,w,z).
 \end{equation}
 For the Grassmann scheme $\Gr_q(A+B,A)$, the number of vertices, the valency, and the multiplicity are respectively given by 
\[g_{A,B}={A+B\brack A}_q,\qquad
 \kappa_{A,B}(\xi)=q^{\xi^2}{A\brack\xi}_q{B\brack\xi}_q,\qquad 
 \mu_{A,B}(h)={A+B\brack h}_q-{A+B\brack h-1}_q.\]
 Denote $D_r=K_r(0;i,i;q)=\mathsf R_q(i,i;r)$ for $0\leq r\leq i$ and
$c_i(v)=2i-a+v$. 

We first give the recurrence relation of the normalized  $q$-Krawtchouk polynomial $L_r(v,w,z)$ involved in the spherical function $\phi_{m,n,r}(v,w,z)$.
\begin{lemma}\label{lem:worth}
Define a normalization of the $q$-Krawtchouk polynomial
\[L_r(v,w,z)
 :=\begin{cases}
 \dfrac{K_r(z;w,c_i(v);q)}{D_r},&r\leq\min\{w,c_i(v)\},\\
 0,&r>\min\{w,c_i(v)\}.
 \end{cases}\] 
 For $0\leq r\leq i$, the normalized $q$-Krawtchouk 
$L_r(v,w,z)$ on $\mathcal I_i$ satisfy
\begin{equation}\label{eqn:response}
L_1L_r
=
A_r^+L_{r+1}
+
A_r^0(v,w)L_r
+
A_r^-(v,w)L_{r-1},
\end{equation}
where the terms
$L_{-1}$ and $L_{i+1}$ are omitted and, for the remaining terms,
\begin{align}\label{eqn:grip}
A_r^+
&=\frac{D_{r+1}}{D_1D_r}
q^r[r+1]_q,   \nonumber  \\
A_r^0(v,w)
&=
\frac{[r]_q}{D_1}
\left(
q^w+q^{c_i(v)}-q^r-q^{r-1}-1
\right),\\
A_r^-(v,w)
&=
\frac{D_{r-1}}{D_1D_r}
q^{2r-2}(q-1)
[w-r+1]_q[c_i(v)-r+1]_q. \nonumber 
\end{align}
\end{lemma}

\begin{proof}
Fix coordinates $v$ and $w$ in the relation index $(v,w,z)$ and denote
$c=c_i(v)$.  For $0\leq h,z\leq\min\{w,c\}$, the polynomials
$K_h(z;w,c;q)$ are the generalized $q$-Krawtchouk polynomials
corresponding to the bilinear-forms scheme $H_q(w,c)=
 \bigl(\operatorname{Mat}_{w\times c}(\F_q),
       \{R_z\}_{z=0}^{\min\{w,c\}}\bigr),$
where the relation $R_z$ is defined by
\[
 (M,N)\in R_z\Longleftrightarrow\operatorname{rank}(M-N)=z.
\]
More precisely, its first eigenvalue is $ P_h(z)=K_h(z;w,c;q).$
It is known that the bilinear-forms scheme $H_q(w,c)$ is $P$-polynomial and self-dual (see \cite{BBIT2021}). 
Hence we have the three-term recurrence relation
\begin{equation}\label{eqn:captain}
 K_1(z;w,c;q)K_r(z;w,c;q)
 =\sum_{\sigma=-1}^1 b_{r+\sigma}(v,w)K_{r+\sigma}(z;w,c;q).    
\end{equation}
Denote $[h]_q=(q^h-1)/(q-1)$.  Then three coefficients (see also
\cite[Section~5.1]{BKZZ-2024}) are
\begin{align*}
 b_{r+\sigma}(v,w)=\begin{cases}
 q^{2r-2}(q-1)[w-r+1]_q[c_i(v)-r+1]_q &\text{~if~} \sigma=-1,\\
 [r]_q\bigl(q^w+q^{c_i(v)}-q^r-q^{r-1}-1\bigr) &\text{~if~} \sigma=0,\\
q^r[r+1]_q  &\text{~if~} \sigma=1.    
 \end{cases}
\end{align*}
Since $D_h$ is independent of $v,w$, substituting $K_h=D_hL_h$
into \cref{eqn:captain} and dividing by $D_1D_r$ gives $L_1L_r=A_r^+L_{r+1}
       +A_r^0(v,w)L_r+A_r^-(v,w)L_{r-1}$, for
$r\leq\min\{w,c\}$.
The case for $r>\min\{w,c\}$ is clear from the definition of $L_r(v,w,z).$
Therefore, the identity holds throughout $\mathcal I_i$.
\end{proof}
Next we consider two normalized  $q$-Hahn polynomials $H_m^{(r)}(v)$ and $J_n^{(r)}(w)$ involved in the spherical function $\phi_{m,n,r}(v,w,z)$.
\begin{lemma}[{cf.\ \cite[Lemma~5.1]{BKZZ-2024}}]\label{lem:nomination}
For $Z,p\in\mathbb N$ and $0\leq j,\xi \leq\min\{Z,p\}$, define the
normalized $q$-Hahn polynomial $\mathcal E_j^{Z,p}$ by
\begin{align}
\mathcal E_j^{Z,p}(\xi)
&:=\frac{E_j(Z,p,Z,Z-\xi;q^{-1})}
{E_j(Z,p,Z,Z;q^{-1})}=\frac{E_j(p,Z,p,p-\xi;q^{-1})}
{E_j(p,Z,p,p;q^{-1})},
\label{eqn:compartment}
\end{align}
so that $\mathcal E_j^{Z,p}(0)=1$.
For $p\ge1$, $0\le j,\xi\le\min\{Z,p-1\}$, we have 
\begin{equation}\label{eqn:truth}
 \mathcal E_j^{Z,p}(\xi)
=
\rho_{j,0}^{Z,p}\mathcal E_j^{Z,p-1}(\xi)
+
\rho_{j,1}^{Z,p}\mathcal E_{j-1}^{Z,p-1}(\xi),
\end{equation}
where the second term is omitted when $j=0$, and the coefficients are
\begin{align*}
 \rho_{j,0}^{Z,p}
 &:=\frac{[Z+p]_q}{[p]_q}\frac{[p-j]_q}{[Z+p-2j]_q}
       \frac{\mu_{Z,p-1}(j)}{\mu_{Z,p}(j)},\qquad \qquad \qquad 0\leq j\leq\min\{Z,p-1\}\\
 \rho_{j,1}^{Z,p}
 &:=\frac{[Z+p]_q}{[p]_q}
       \frac{q^{p-j+1}[Z-j+1]_q}{[Z+p-2j+2]_q}
       \frac{\mu_{Z,p-1}(j-1)}{\mu_{Z,p}(j)},
       \qquad 1\leq j\leq\min\{Z,p-1\}.
\end{align*}
\end{lemma}
\begin{remark}
The second equality in \cref{eqn:compartment}  is due to $\Gr_q(Z+p,Z)\cong\Gr_q(Z+p,p)$. 
In \cref{lem:nomination}, we set $\rho_{0,1}^{Z,p}=0$ for $p\geq1$. 
When $1\leq p\leq Z$, the recurrence relation in \cref{eqn:truth} extends to $j=p$ by
$\mathcal E_p^{Z,p}(\xi)=\mathcal E_{p-1}^{Z,p-1}(\xi)
$ for $0\leq\xi\leq p-1$,
and we set $\rho_{p,0}^{Z,p}=0$ and $\rho_{p,1}^{Z,p}=1$.     
\end{remark}
\begin{lemma}\label{lem:discriminate}
Let $Z,p\in\mathbb N$ and $0\leq j\leq\min\{Z,p\}$.
For $0\leq\xi\leq\min\{Z,p+1\}$, the normalized $q$-Hahn
polynomials satisfy the following identity, where
$\mathcal E_j^{Z,p}(\xi)$ outside its orthogonality grid denotes
evaluation of the same polynomial expression:
\begin{equation}\label{eqn:interference}
[p+1-\xi]_q \mathcal E_j^{Z,p}(\xi)
=
\lambda_{j,0}^{Z,p}\mathcal E_j^{Z,p+1}(\xi)
+
\lambda_{j,1}^{Z,p}\mathcal E_{j+1}^{Z,p+1}(\xi),
\end{equation}
where 
\[
\lambda_{j,0}^{Z,p}
=
[p+1]_q
\frac{[p-j+1]_q}{[Z+p-2j+1]_q},
\qquad
\lambda_{j,1}^{Z,p}
=
[p+1]_q
\frac{q^{p-j+1}[Z-j]_q}{[Z+p-2j+1]_q}.
\]
\end{lemma}
\begin{proof}
The polynomial $\mathcal E_j^{Z,p}(\xi)$ is the normalized $q$-Hahn polynomial
associated with $\Gr_q(Z+p,p)$, with relation variable $\xi$.
So the valencies are $\kappa_{Z,p}(\xi)
=
q^{\xi^2}
{Z\brack \xi}_q
{p\brack \xi}_q.$
Thus
\begin{equation}\label{eqn:grimace}
 [p+1-\xi]_q\kappa_{Z,p+1}(\xi)
=
[p+1]_q\kappa_{Z,p}(\xi),
\end{equation}
by the identity $[p+1-\xi]_q{p+1\brack \xi}_q
=
[p+1]_q{p\brack \xi}_q.$
Define an inner product
\[
\langle F,G\rangle_{Z,p}
:=
\sum_{\xi=0}^{\min\{Z,p\}}
\kappa_{Z,p}(\xi)F(\xi)G(\xi).
\]
The orthogonality relations give $\langle
\mathcal E_j^{Z,p},\mathcal E_k^{Z,p}
\rangle_{Z,p}
=
\frac{g_{Z,p}}{\mu_{Z,p}(j)}\delta_{j,k}.$
Since $[p+1-\xi]_q\mathcal E_j^{Z,p}$ has degree at most $j+1$,
it can be expressed in $\mathcal E_0^{Z,p+1},\ldots,
\mathcal E_{\min\{j+1,Z,p+1\}}^{Z,p+1}$.
For $k<j$, \cref{eqn:grimace,eqn:truth} give
\begin{align*}
\langle
[p+1-\xi]_q\mathcal E_j^{Z,p},\mathcal E_k^{Z,p+1}
\rangle_{Z,p+1}
&=
[p+1]_q
\langle
\mathcal E_j^{Z,p},\mathcal E_k^{Z,p+1}
\rangle_{Z,p}\\
&=
[p+1]_q
\langle
\mathcal E_j^{Z,p},
\rho_{k,0}^{Z,p+1}\mathcal E_k^{Z,p}
+
\rho_{k,1}^{Z,p+1}\mathcal E_{k-1}^{Z,p}
\rangle_{Z,p}
=0.
\end{align*}
Taking the inner product with $\mathcal E_{j+\delta}^{Z,p+1}$
for each $\delta\in\{0,1\}$ for which $j+\delta$ is admissible, and again using
\cref{eqn:grimace}, we obtain
\begin{align*}
\lambda_{j,\delta}^{Z,p}
=
\frac{
\langle
[p+1-\xi]_q\mathcal E_j^{Z,p},
\mathcal E_{j+\delta}^{Z,p+1}
\rangle_{Z,p+1}
}{
\|\mathcal E_{j+\delta}^{Z,p+1}\|_{Z,p+1}^{2}
}
&=
[p+1]_q
\rho_{j+\delta,\delta}^{Z,p+1}
\frac{\|\mathcal E_j^{Z,p}\|_{Z,p}^{2}}
     {\|\mathcal E_{j+\delta}^{Z,p+1}\|_{Z,p+1}^{2}}\\
&=
\frac{[p+1]_q^2}{[Z+p+1]_q}
\frac{\mu_{Z,p+1}(j+\delta)}{\mu_{Z,p}(j)}
\rho_{j+\delta,\delta}^{Z,p+1},
\end{align*}
where the last equality follows from $\|\mathcal E_j^{Z,p}\|_{Z,p}^{2}
=
\frac{g_{Z,p}}{\mu_{Z,p}(j)}$ and $\frac{g_{Z,p}}{g_{Z,p+1}}
=
\frac{[p+1]_q}{[Z+p+1]_q}.$
Substituting the coefficients $\rho_{j,0}^{Z,p}$ and
$\rho_{j,1}^{Z,p}$ from \cref{lem:nomination} completes the proof.
\end{proof}
Put $A=a-i$,  $B=b-i$ and $p=i-r$.
Define two normalizations of $q$-Hahn polynomials
\[H_m^{(r)}(v):=\mathcal E_m^{A,p}(A-v), \quad J_n^{(r)}(w):=\mathcal E_n^{B,p}(i-w).\]
For $r<i$, \cref{lem:nomination} gives
\begin{align}\label{eqn:bland}
 H_m^{(r)}
 =\rho_{m,0}^{A,p}H_m^{(r+1)}
   +\rho_{m,1}^{A,p}H_{m-1}^{(r+1)}, \qquad
 J_n^{(r)}
 =\rho_{n,0}^{B,p}J_n^{(r+1)}
   +\rho_{n,1}^{B,p}J_{n-1}^{(r+1)}.
\end{align}
For $r>0$, 
\cref{lem:discriminate} gives
\begin{align}
[c_i(v)-r+1]_qH_m^{(r)}
&=\lambda_{m,0}^{A,p}H_m^{(r-1)}
  +\lambda_{m,1}^{A,p}H_{m+1}^{(r-1)},
  \label{eqn:draft}\\
[w-r+1]_qJ_n^{(r)}
&=\lambda_{n,0}^{B,p}J_n^{(r-1)}
  +\lambda_{n,1}^{B,p}J_{n+1}^{(r-1)}.
  \label{eqn:worker}
\end{align}
For $\sigma\in\{-1,0,1\}$ and $p=i-r$, define
\begin{align*}
\mathsf h_{m,\sigma}^{(r)}
&=\frac{\mu_{A,p}(m+\sigma)}{g_{A,p}}
 \sum_{\xi=0}^{\min\{A,p\}}\kappa_{A,p}(\xi)
 H_1^{(0)}(A-\xi)H_m^{(r)}(A-\xi)
 H_{m+\sigma}^{(r)}(A-\xi),\\
\mathsf j_{n,\sigma}^{(r)}
&=\frac{\mu_{p,B}(n+\sigma)}{g_{p,B}}
 \sum_{\eta=0}^{\min\{p,B\}}\kappa_{p,B}(\eta)
 J_1^{(0)}(i-\eta)J_n^{(r)}(i-\eta)
 J_{n+\sigma}^{(r)}(i-\eta).
\end{align*}
Orthogonality of $q$-Hahn polynomials gives the three-term expansions
\begin{equation}\label{eqn:braid}
H_1^{(0)}H_m^{(r)}
=\sum_{\sigma=-1}^{1}\mathsf h_{m,\sigma}^{(r)}H_{m+\sigma}^{(r)},
\qquad
J_1^{(0)}J_n^{(r)}
=\sum_{\sigma=-1}^{1}\mathsf j_{n,\sigma}^{(r)}J_{n+\sigma}^{(r)}.
\end{equation}

\begin{lemma}\label{lem:wheat}
Let $(m,n,r)\in\mathcal D_i^*$ and 
$\Omega_t=\{(v,w,z)\in\mathcal I_i:t\leq\min\{w,c_i(v)\}\}$
for $0\leq t\leq i$.
 The coefficients in
\cref{lem:worth}, after multiplication by $H_m^{(r)}J_n^{(r)}$,
have the following expansions 
\begin{align*}
A_r^+H_m^{(r)}J_n^{(r)}
&=\sum_{\delta,\eta\in\{0,1\}}
\mathsf u_{\delta,\eta}(m,n,r)
H_{m-\delta}^{(r+1)}J_{n-\eta}^{(r+1)}
\qquad\text{on}\quad \Omega_{r+1}, \\
A_r^0(v,w) H_m^{(r)}J_n^{(r)}
&=\sum_{(s,t)\in\mathcal M(m,n)}
\mathsf v_{s,t}(m,n,r)
H_s^{(r)}J_t^{(r)}
\qquad \ \text{on} \quad\Omega_r,\\
A_r^-(v,w) H_m^{(r)}J_n^{(r)}
&=\sum_{\delta,\eta\in\{0,1\}}
\mathsf w_{\delta,\eta}(m,n,r)
H_{m+\delta}^{(r-1)}J_{n+\eta}^{(r-1)}
\qquad\text{on} \quad\Omega_{r-1},
\end{align*}
where $\mathcal M(m,n)
=
\{(m,n),(m\pm 1,n),(m,n\pm 1)\}$ and the coefficients are given in the proof.
\end{lemma}
\begin{proof}

First assume $r<i$ and put $p=i-r$.  The normalized $q$-Hahn
polynomials $H_m^{(r)}$ and $J_n^{(r)}$ correspond to $\Gr_q(A+p,A)$ and $\Gr_q(B+p,p)$,
respectively.  
The isomorphism $\Gr_q(A+p,A)\cong\Gr_q(A+p,p)$ gives the form used in \cref{lem:nomination}. 
 Since $A_r^+$ is independent of $v,w$, the relations \cref{eqn:bland} give the expression on $\Omega_{r+1}$
\[\mathsf u_{\delta,\eta}(m,n,r)=A_r^+\rho_{m,\delta}^{A,i-r}\rho_{n,\eta}^{B,i-r}.\] 
Note that the normalized $q$-Hahn polynomials $H_1^{(0)}(v)$ and $J_1^{(0)}(w)$ are polynomials of
degree one in $q^v$ and $q^w$, respectively. 
More precisely, 
\begin{align*}
H_1^{(0)}(v)
=
\frac{[A+i]_q[v]_q-[A]_q^2}{q^A[A]_q[i]_q},    \qquad 
J_1^{(0)}(w)
=
\frac{[B+i]_q[w]_q-[i]_q^2}{q^i[i]_q[B]_q}.
\end{align*}
We can obtain the expressions of $q^v$ and $q^w$ in terms of $H_1^{(0)}(v)$ and $J_1^{(0)}(w)$, respectively
\begin{align*}
q^v
&=
1+(q-1)\frac{[A]_q^2+q^A[A]_q[i]_q H_1^{(0)}(v)}{[A+i]_q}:=\gamma_H+(q^A-\gamma_H)H_1^{(0)}(v),\\
q^w&=1+(q-1)
\frac{[i]_q^2+q^i[i]_q[B]_q J_1^{(0)}(w)}{[B+i]_q}:=
 \gamma_J+(q^i-\gamma_J)J_1^{(0)}(w).         
 \end{align*}
Since $c_i(v)=2i-a+v$, it follows from \cref{eqn:grip} that
\begin{equation}\label{eqn:creep}
A_r^0(v,w)
 =
 \frac{[r]_q}{D_1} \left(C_r+B_r^H H_1^{(0)}(v)+B_r^J J_1^{(0)}(w)\right),    
\end{equation}
 where $C_r=
 \gamma_J+q^{2i-a}\gamma_H-q^r-q^{r-1}-1$, $B_r^H=
 q^{2i-a}(q^A-\gamma_H)$ and $B_r^J=q^i-\gamma_J.$
Applying \cref{eqn:braid} to
\cref{eqn:creep} on $\Omega_r$ gives 
\[\mathsf v_{s,t}(m,n,r)
 =\frac{[r]_q}{D_1}\bigl(
  C_r\delta_{s,m}\delta_{t,n}
  +B_r^H\mathsf h_{m,s-m}^{(r)}\delta_{t,n}
  +B_r^J\mathsf j_{n,t-n}^{(r)}\delta_{s,m}\bigr).\]
For $r>0$, put $p=i-r$ and let
$\kappa_r:=\frac{D_{r-1}}{D_1D_r}\,q^{2r-2}(q-1)$.
 Again using \cref{eqn:grip}, we have 
\[
A_r^-(v,w)
=\kappa_r[c_i(v)-r+1]_q[w-r+1]_q.
\]
On $\Omega_{r-1}$, \cref{lem:discriminate} gives
\[
\begin{aligned}
A_r^-(v,w)H_m^{(r)}J_n^{(r)}
=\kappa_r
\left(\sum_{\delta=0}^1
\lambda_{m,\delta}^{A,p}H_{m+\delta}^{(r-1)}\right)
\left(\sum_{\eta=0}^1
\lambda_{n,\eta}^{B,p}J_{n+\eta}^{(r-1)}\right)=\sum_{\delta,\eta\in\{0,1\}}
\kappa_r\lambda_{m,\delta}^{A,i-r}
\lambda_{n,\eta}^{B,i-r}
H_{m+\delta}^{(r-1)}J_{n+\eta}^{(r-1)}.
\end{aligned}
\]
Therefore, we have
\[\mathsf w_{\delta,\eta}(m,n,r) =\frac{D_{r-1}}{D_1D_r}q^{2r-2}(q-1)  \lambda_{m,\delta}^{A,i-r}\lambda_{n,\eta}^{B,i-r}.\]
This completes the proof.
\end{proof}
Now we obtain the trivariate analogue of the usual
tridiagonal condition for the Krein matrix of a $Q$-polynomial association scheme.

\begin{lemma}
\label{lem:monster}
For $1\leq i\leq a-1$, fix $\alpha=(m,n,r)\in\mathcal D_i^*$. 
Let  $\mathcal N_s^Q(\alpha)=
\{\alpha,\alpha\pm\epsilon_s\}$ for $s=1,2$, and 
\[\begin{aligned}
\mathcal N_3^Q(\alpha)={}&\{\alpha,\alpha\pm\epsilon_1,\alpha\pm\epsilon_2\} 
\cup\{\alpha-\delta\epsilon_1-\eta\epsilon_2+\epsilon_3, \alpha+\delta\epsilon_1+\eta\epsilon_2-\epsilon_3 \mid
             \delta,\eta\in\{0,1\}\}.
\end{aligned}\]
Then, for $s=1,2,3$,  the Krein parameter $q_{\epsilon_s,\alpha}^{\beta}\neq0$ implies $\beta \in \mathcal N_s^Q(\alpha) \cap \mathcal D_i^*$.
Consequently, $q_{\epsilon_s,\alpha}^{\beta}\neq0$ implies that 
 $\beta\le_{312}\alpha+\epsilon_s.$
 Moreover, $q_{\epsilon_s,\alpha}^{\alpha+\epsilon_s}\neq 0$ whenever $\alpha, \alpha+\epsilon_s\in\mathcal D_i^*$.

\end{lemma}

\begin{proof}
Put $A=a-i$ and $B=b-i$.
Then \cref{eqn:calculation} gives
\begin{equation}\label{eqn:related}
 \phi_{m,n,r}(v,w,z)
 =H_m^{(r)}(v)J_n^{(r)}(w)L_r(v,w,z).
\end{equation}
When $r>\min\{w,c_i(v)\}$, the right-hand side equals zero.
 For $s=1,2,3$, $\alpha,\beta\in\mathcal D_i^*$, and
 $\tau=(v,w,z)\in\mathcal I_i$, the definition of the Krein
 parameters and \cref{eqn:finish} give
 \begin{equation}\label{eqn:island}
 \phi_{\epsilon_s}(\tau) \phi_\alpha(\tau)=\frac{1}{d_{\epsilon_s} d_\alpha} \sum_{\beta\in \mathcal D_i^\ast} q_{\epsilon_s,\alpha}^{\beta} d_\beta \phi_\beta(\tau).
 \end{equation}
 Since $\phi_{\epsilon_1}(\tau)=H_1^{(0)}(v)$ and
 $\phi_{\epsilon_2}(\tau)=J_1^{(0)}(w)$,
 \cref{eqn:braid} gives
 \begin{align} 
&\phi_{\epsilon_1}(\tau) \phi_\alpha(\tau)=H_1^{(0)}(v) H_m^{(r)}(v)J_n^{(r)}(w)L_r(v,w,z) 
 =  \sum_{\sigma=-1}^{1}\mathsf h_{m,\sigma}^{(r)} \phi_{m+\sigma,n,r}(\tau), \label{eqn:branch}\\ 
& \phi_{\epsilon_2}(\tau) \phi_\alpha(\tau)=\sum_{\sigma=-1}^{1}\mathsf j_{n,\sigma}^{(r)}
 J_{n+\sigma}^{(r)} H_m^{(r) } L_r(v,w,z)  =  \sum_{\sigma=-1}^{1}\mathsf j_{n,\sigma}^{(r)} \phi_{m,n+\sigma,r}(\tau).\label{eqn:marine}
 \end{align}
 Since $\phi_{\epsilon_3}(\tau)=L_1(v,w,z)$,
 multiplying the three expansions of \cref{lem:wheat} by $L_{r+1}$, $L_r$, and $L_{r-1}$, respectively, and using \cref{lem:worth}, we obtain 
 \begin{align}\label{eqn:linen}
 \phi_{\epsilon_3}^{(i)}\phi_{m,n,r}^{(i)}
 &=
 \sum_{\delta,\eta\in\{0,1\}}
 \mathsf u_{\delta,\eta}(m,n,r)
 \phi_{m-\delta,n-\eta,r+1}^{(i)}
 +\sum_{(s,t)\in\mathcal M(m,n)}
 \mathsf v_{s,t}(m,n,r)\phi_{s,t,r}^{(i)}\notag\\
 &\qquad+\sum_{\delta,\eta\in\{0,1\}}
 \mathsf w_{\delta,\eta}(m,n,r)
 \phi_{m+\delta,n+\eta,r-1}^{(i)}.            
\end{align}
Since $L_t$
vanishes outside $\Omega_t$, this identity holds on all of
$\mathcal I_i$.
 Comparing
\cref{eqn:island} with
\cref{eqn:branch,eqn:marine,eqn:linen} yields, for
$\sigma\in\{-1,0,1\}$ and whenever the index belongs to $\mathcal D_i^*$,
\begin{align}
 q_{\epsilon_1,(m,n,r)}^{(m+\sigma,n,r)}
 =\frac{d_{\epsilon_1}d_{m,n,r}} 
 {d_{m+\sigma,n,r}}\,
 \mathsf h_{m,\sigma}^{(r)},                        
 \qquad 
 q_{\epsilon_2,(m,n,r)}^{(m,n+\sigma,r)}
 =\frac{d_{\epsilon_2}d_{m,n,r}}
 {d_{m,n+\sigma,r}}\,
 \mathsf j_{n,\sigma}^{(r)}.   \label{eqn:camera}
\end{align}
\begin{equation}\label{eqn:salesperson}
 q_{\epsilon_3,\alpha}^{\beta}
 =\frac{d_{\epsilon_3}d_\alpha}{d_\beta}
 \begin{cases}
  \mathsf u_{\delta,\eta}(m,n,r),
   &\beta=(m-\delta,n-\eta,r+1),\\
  \mathsf v_{s,t}(m,n,r),
   &\beta=(s,t,r),\ (s,t)\in\mathcal M(m,n),\\
  \mathsf w_{\delta,\eta}(m,n,r),
   &\beta=(m+\delta,n+\eta,r-1),\\
  0,&\text{otherwise}.
 \end{cases}
\end{equation}
Thus 
$q_{\epsilon_s,\alpha}^{\beta}\ne0$ implies
$\beta\in\mathcal N_s^Q(\alpha)\cap\mathcal D_i^*$ for all
$s=1,2,3$.  
Moreover, $\beta\in\mathcal N_s^Q(\alpha)\cap\mathcal D_i^*$ gives $\beta\le_{312}\alpha+\epsilon_s$.

Finally, suppose $(m+1,n,r)\in\mathcal D_i^*$.  Comparing the
leading coefficients of both sides in the first expansion in \cref{eqn:braid}
gives $\mathsf h_{m,1}^{(r)}\neq 0$.
The same argument applied to the second
expansion in \cref{eqn:braid} gives $\mathsf j_{n,1}^{(r)}\neq 0$
whenever $(m,n+1,r)\in\mathcal D_i^*$.  
If $(m,n,r+1)\in\mathcal D_i^*$, then $m\leq\min\{A,i-r-1\}$ and $n\leq\min\{B,i-r-1\}.$
Consequently the coefficients $\rho_{m,0}^{A,i-r}$ and
$\rho_{n,0}^{B,i-r}$ are strictly positive.  Since $A_r^+>0$, the formula in \cref{lem:wheat}
gives 
\[
 \mathsf u_{0,0}(m,n,r)
 =A_r^+\rho_{m,0}^{A,i-r}\rho_{n,0}^{B,i-r}\neq 0.
\]
Since all multiplicities $d_\gamma$ are positive, \cref{eqn:camera,eqn:salesperson} imply $q_{\epsilon_s,\alpha}^{\alpha+\epsilon_s}\neq0$ whenever
$\alpha, \alpha+\epsilon_s\in\mathcal D_i^*$.
\end{proof}

\begin{theorem}\label{thm:principle}
For $1\leq i\leq a-1$, the association scheme induced on the Grassmann
shell $X_i$ is trivariate  $Q$-polynomial on $\mathcal D_i^*$ with respect to $\le_{312}$.
\end{theorem}

\begin{proof}
We first verify that $\mathcal D_i^*$ is a down-set w.r.t. $\preceq_{cw}$. 
Let $(m,n,r)\in\mathcal D_i^*$
and suppose that
\[
0\leq m'\leq m,\qquad
0\leq n'\leq n,\qquad
0\leq r'\leq r.
\]
Since $r'\leq r$, we have $i-r\leq i-r'$. Consequently,
\[
m'
\leq m
\leq\min\{a-i,i-r\}
\leq\min\{a-i,i-r'\},\quad 
n'
\leq n
\leq\min\{i-r,b-i\}
\leq\min\{i-r',b-i\}.
\]
Thus $(m',n',r')\in\mathcal D_i^*$, so $\mathcal D_i^*$ is a
down-set w.r.t. $\preceq_{cw}$, as required in \cref{def:expertise}.  The assumptions
$1\leq i\leq a-1$ and
$a\leq b$ also imply that
$\epsilon_1,\epsilon_2,\epsilon_3\in\mathcal D_i^*$.
By \cref{lem:monster}, the Krein parameters satisfy
\[
q_{\epsilon_s,\alpha}^{\beta}\neq0
\quad\Longrightarrow\quad
\beta\leq_{312}\alpha+\epsilon_s,
\]
and $q_{\epsilon_s,\alpha}^{\alpha+\epsilon_s}\neq 0$ 
whenever $\alpha+\epsilon_s\in\mathcal D_i^*$.
It follows from \cref{prop:substitute} that the association scheme on
$X_i$ is trivariate $Q$-polynomial on $\mathcal D_i^*$ with
respect to $\leq_{312}$.
\end{proof}
\begin{remark}
Although the intersection numbers can be obtained from the spherical functions in principle, the resulting formulas involve sums of products of spherical functions and do not directly give their support and nonvanishing properties. 
The combinatorial rank arguments and constructions used for the $P$-polynomial structure can provide these
properties directly.
\end{remark}

\section[Polyhedral Leonard pairs]{Polyhedral $A_M$-Leonard pairs}\label{sect:commemorate}
The root-system viewpoint for multivariate $P$- and $Q$-polynomial
association schemes was emphasized in \cite[Section~5(ii)]{BKZZ-2025},
following the work of Iliev and Terwilliger \cite{IT-2012}. In
\cite[Section~7]{BKZZ-2024}, $A_M$-Leonard pairs are obtained from
multivariate $P$- and $Q$-polynomial association schemes whose index domain is the standard simplex.
In this section, we extend this construction to polyhedral index domains and show that the association scheme induced on an interior Grassmann shell provides a polyhedral $A_3$-Leonard pair. 

Fix a positive integer $M$ and a bounded polyhedron
$\mathcal P\subseteq\mathbb R_{\geq 0}^M$.  
Let $\mathcal D=\mathcal P\cap\mathbb N^M,$ and assume that $\mathcal D$ is a down-set with respect to $\preceq_{cw}$.  
Choose linearly independent vectors $\alpha_1,\ldots,\alpha_M\in\mathbb Z^M$, and take
\begin{equation}\label{eqn:conflict}
 \Phi=\{\pm(\alpha_j+\alpha_{j+1}+\cdots+\alpha_k):
1\leq j\leq k\leq M\}.   
\end{equation}
Moreover, assume  that $\epsilon_1,\ldots,\epsilon_M\in\mathcal D\cap\Phi.$
For $\alpha,\beta\in\mathcal D$, we say that $\alpha$ and $\beta$ are \emph{$\Phi$-adjacent}, and write $\alpha\sim_\Phi\beta$ if $\beta-\alpha\in\{\mathbf 0\}\cup\Phi.$

The choice of $\Phi$ is motivated by the root system of type $A_M$.
Let $\gamma_1,\ldots,\gamma_M$ be the standard simple roots of $A_M$.  
Its positive roots are the consecutive sums
\[
\gamma_j+\gamma_{j+1}+\cdots+\gamma_k
\qquad (1\leq j\leq k\leq M).
\]
Hence the invertible linear map determined by $\gamma_j\longmapsto\alpha_j, \ (1\leq j\leq M)$
identifies $\Phi$ with a generally nonstandard linear realization of $A_M$.

Following
\cite{IT-2012,BKZZ-2024}, we extend the notion of an
$A_M$-Leonard pair to this polyhedral domain.
\begin{definition}[{cf.\ \cite[Definition~7.1]{BKZZ-2024}}]
\label{def:pie}
Let $\F$ be a field, let $V$ be a finite-dimensional vector space
over $\F$, and let $H,\widetilde H\subseteq\operatorname{End}(V)$
be $M$-dimensional subspaces.  We call $(H,\widetilde H)$ a
 \emph{polyhedral $A_M$-Leonard pair on $\mathcal D$ with respect to
 $\Phi$} if the following conditions hold.
 \begin{enumerate}[label={\rm(\roman*)}]
  \item The elements of $H$ are diagonalizable and mutually commute.
  \item The elements of $\widetilde H$ are diagonalizable and mutually commute.
  \item There exists a bijection $\alpha\mapsto V_\alpha$ from
  $\mathcal D$ to the common eigenspaces of $H$ such that
  \[
   \widetilde H V_\alpha
   \subseteq
   \sum_{\substack{\beta\in\mathcal D\\
                     \alpha\sim_\Phi\beta}}V_\beta
   \qquad(\alpha\in\mathcal D).
  \]
  \item There exists a bijection $\alpha\mapsto\widetilde V_\alpha$
  from $\mathcal D$ to the common eigenspaces of $\widetilde H$ such
  that
  \[
   H\widetilde V_\alpha
   \subseteq
   \sum_{\substack{\beta\in\mathcal D\\
                     \alpha\sim_\Phi\beta}}\widetilde V_\beta
   \qquad(\alpha\in\mathcal D).
  \]
  \item There is no subspace $W$ of $V$ such that
$HW\subseteq W,\quad \widetilde H W\subseteq W,
   \quad 0\neq W\neq V.$
  \item For every $\alpha\in\mathcal D$,
  $\dim V_\alpha=\dim\widetilde V_\alpha=1$.
  \item There exists a nondegenerate symmetric bilinear form
  $\langle\ ,\ \rangle$ on $V$ such that
  \[
   \langle V_\alpha,V_\beta\rangle=0,
   \qquad
   \langle\widetilde V_\alpha,\widetilde V_\beta\rangle=0
   \qquad(\alpha\neq\beta).
  \]
 \end{enumerate}
\end{definition}

\begin{definition}
\label{def:account}
Let $\mathfrak X=(X,\{A_\alpha\}_{\alpha\in\mathcal D})$
be a symmetric association scheme, with primitive idempotents $\{E_\alpha\}_{\alpha\in\mathcal D}$.
We call $\mathfrak X$ a polyhedral $A_M$ multivariate $P$- and $Q$-polynomial association scheme on $\mathcal D$ with respect to $\Phi$ if the following conditions hold.
\begin{enumerate}[label={\rm(\roman*)}]
 \item
 $\mathfrak X$ is an $M$-variate $P$-polynomial association scheme on
 $\mathcal D$ with respect to a monomial order $\le_P$, and, for
 every $1\leq j\leq M$ and $\alpha,\beta\in\mathcal D$, if $p_{\epsilon_j,\alpha}^{\beta}\neq0$ then $\beta-\alpha\in\{\mathbf 0\}\cup\Phi.$
 \item
 $\mathfrak X$ is an $M$-variate $Q$-polynomial association scheme on
 $\mathcal D$ with respect to a monomial order $\le_Q$, and, for
 every $1\leq j\leq M$ and $\alpha,\beta\in\mathcal D$, if $q_{\epsilon_j,\alpha}^{\beta}\neq0$ then $\beta-\alpha\in\{\mathbf 0\}\cup\Phi.$
\end{enumerate}
\end{definition}
\begin{remark}
If we take $\mathcal D=\{\alpha\in\mathbb N^M:|\alpha|\leq N\}$
and choose $\alpha_j=\epsilon_j-\epsilon_{j+1}\ (1\leq j<M)$,  $\alpha_M=\epsilon_M$, then \cref{eqn:conflict} becomes
\[
 \Phi=\{\pm\epsilon_j:1\leq j\leq M\}
 \cup\{\epsilon_j-\epsilon_k:1\leq j\neq k\leq M\}.
\]
Hence
\cref{def:pie,def:account} reduce to Definitions~7.1 and 7.2 in \cite{BKZZ-2024}, respectively.
\end{remark}

The following result extends \cite[Theorem~7.3]{BKZZ-2024} to a polyhedral domain with a more general adjacency relation.
\begin{theorem}
\label{thm:obligation}
A polyhedral $A_M$ multivariate $P$- and
$Q$-polynomial association scheme on $\mathcal D$ has the structure of a polyhedral $A_M$-Leonard
pair over $\mathbb C$.
\end{theorem}

\begin{proof}
Fix $x_0\in X$, and let $\mathcal A$ and $\mathcal A^*=\mathcal A^*(x_0)$ be the Bose--Mesner algebra and the dual Bose--Mesner algebra, respectively.  
Let
\[
 V=\mathcal A\hat{x}_0=\mathcal A^*\mathbf1,
 \qquad
 H=\operatorname{span}_{\mathbb C}
 \{A_{\epsilon_1},\ldots,A_{\epsilon_M}\},
 \qquad
 \widetilde H=\operatorname{span}_{\mathbb C}
 \{A_{\epsilon_1}^*,\ldots,A_{\epsilon_M}^*\}.
\]
For $\alpha\in\mathcal D$, denote $V_\alpha=E_\alpha V=\operatorname{span}\{v_\alpha\}$, $\widetilde V_\alpha=E_\alpha^*V$ and $v_\alpha=E_\alpha\hat{x}_0.$
The $P$- and $Q$-polynomial properties imply that $H$ and $\widetilde H$ generate $\mathcal A|_V$ and $\mathcal A^*|_V$, respectively. 
Their elements are diagonalizable and mutually commuting, so conditions~(i) and~(ii) of \cref{def:pie} hold. 
The spaces $V_\alpha$ and $\widetilde V_\alpha$ are the common eigenspaces of $H$ and $\widetilde H$, respectively, and are one-dimensional.
Thus condition~(vi) holds.  
The standard identities for intersection numbers and Krein parameters give 
\[
 A_{\epsilon_s}\widetilde V_\alpha \subseteq
 \sum_{\substack{\beta\in\mathcal D\\
    p_{\epsilon_s,\alpha}^{\beta}\neq0}} \widetilde V_\beta,
 \qquad
 A_{\epsilon_s}^*V_\alpha \subseteq
 \sum_{\substack{\beta\in\mathcal D\\
    q_{\epsilon_s,\alpha}^{\beta}\neq0}} V_\beta.
\]
Then the two conditions in \cref{def:account} imply the following
\[
 H\widetilde V_\alpha \subseteq
 \sum_{\substack{\beta\in\mathcal D\\
\beta-\alpha\in\{\mathbf0\}\cup\Phi}} \widetilde V_\beta,
 \qquad
 \widetilde H V_\alpha \subseteq
 \sum_{\substack{\beta\in\mathcal D\\
 \beta-\alpha\in\{\mathbf0\}\cup\Phi}} V_\beta,
\]
which are conditions~(iii) and~(iv) of \cref{def:pie}. 
Since $\mathfrak X$ is symmetric by assumption in \cref{def:account}, condition~(vii) holds as well.

Finally, let $0\ne W\subseteq V$ be invariant under $H$ and
$\widetilde H$. Since $H$ generates $\mathcal A|_V$, we have $W=\bigoplus_{\alpha\in\Lambda}V_\alpha$ for some nonempty $\Lambda\subseteq\mathcal D.$
The $Q$-polynomial property gives
$q_{\epsilon_s,\alpha}^{\alpha+\epsilon_s}\ne0$. 
Since $d_\beta q_{\epsilon_s,\alpha}^{\beta}
=d_\alpha q_{\epsilon_s,\beta}^{\alpha},$ we also have $q_{\epsilon_s,\alpha+\epsilon_s}^{\alpha}\ne0$.
Thus invariance under $\widetilde H$ implies that
$\alpha\in\Lambda$ if and only if  $\alpha\pm \epsilon_s \in\Lambda$.
Since $\mathcal D$ is a down-set w.r.t. $\preceq_{cw}$, repeated this equivalence gives that  $\alpha\in\Lambda\iff 0\in\Lambda$.
As $\Lambda\ne\varnothing$, we obtain $\Lambda=\mathcal D$ and
$W=V$, which proves (v).

Therefore $(H,\widetilde H)$ is a polyhedral $A_M$-Leonard pair.
\end{proof}

\begin{corollary}
\label{cor:hierarchy}
Let $1\leq i\leq a-1$. The association scheme induced on each
interior Grassmann shell $X_i$ is a polyhedral $A_3$ multivariate
$P$- and $Q$-polynomial association scheme and carries a polyhedral
$A_3$-Leonard pair over $\mathbb C$.
\end{corollary}

\begin{proof}
 Note that the index sets for $\widehat  A_\alpha$ and $E_\alpha$ coincide, namely, 
\[
\mathcal D_i=\mathcal D_i^*
=\{(u,v,w)\in\mathbb N^3 \mid 
u\leq a-i,\ v\leq b-i,\ u+w\leq i,\ v+w\leq i\}.
\]
Thus the adjacency matrices and primitive idempotents are indexed by
the same bounded polyhedral domain.  
Take three linearly independent vectors
$\alpha_1=\epsilon_1-\epsilon_3$, $\alpha_2=\epsilon_3$, and
$\alpha_3=\epsilon_2-\epsilon_3$. 
Denote $\Phi^+
=\{\epsilon_1,\epsilon_2,\epsilon_3,
\epsilon_1-\epsilon_3,\epsilon_2-\epsilon_3,
\epsilon_1+\epsilon_2-\epsilon_3\}$ and $\Phi=\Phi^+\sqcup(-\Phi^+)$.
For the $P$-polynomial structure, \cref{lem:plain} shows that
\[
\widehat p_{\epsilon_s,\alpha}^{\beta}\neq0
\quad\Longrightarrow\quad
\beta-\alpha\in
\{0,\pm\epsilon_3,\pm\epsilon_1,\pm\epsilon_2,
\pm(\epsilon_1-\epsilon_3),\pm(\epsilon_2-\epsilon_3)\}
\subseteq\{\mathbf0\}\cup\Phi.
\]
For the $Q$-polynomial structure, \cref{lem:monster} shows that $\widehat q_{\epsilon_s,\alpha}^{\beta}\neq0
\Longrightarrow
\beta-\alpha\in\{\mathbf0\}\cup\Phi.$
Thus the two conditions with respect to $\Phi$ in \cref{def:account} hold.
Together with \cref{thm:laboratory,thm:principle},  the interior Grassmann shell is a polyhedral $A_3$ multivariate $P$- and $Q$-polynomial association scheme.  
The Leonard-pair conclusion follows from \cref{thm:obligation}.   
\end{proof}

\section{Concluding remarks}

In this paper, we show that the association scheme induced on each interior shell of the Grassmann scheme is both trivariate $P$-polynomial and
trivariate $Q$-polynomial.  
The two assertions arise, respectively, from combinatorial arguments for the intersection numbers and from Dunkl's spherical functions together with identities and recurrence relations for the associated $q$-Hahn and $q$-Krawtchouk polynomials.  
Thus the interior shells provide a family of trivariate examples with a common bounded non-simplicial polyhedral index domain.  
We also extend the notion of an $A_M$-Leonard pair from standard simplicial domains to polyhedral domains and, at the same time, generalize the standard $A_M$ adjacency relation to one determined by $\Phi$.  
We then show that each interior shell of Grassmann scheme carries a polyhedral $A_3$-Leonard pair with respect to $\Phi$.

\section*{Acknowledgments}
     This research is supported by National Natural Science Foundation of China No.\ 12571353 and Natural Science Foundation of Shanghai No.\ 24ZR1455100.
\bibliography{reference}
\end{document}